\documentclass[12pt,a4paper]{amsart}
\usepackage{amssymb}
\usepackage{epsfig}
\usepackage{graphicx}
\newcommand{\be}{\begin{enumerate}\item}
\newcommand{\ee}{\end{enumerate}}
\newtheorem{thm}{Theorem}[section]
\newtheorem{cor}[thm]{Corollary}
\newtheorem{lem}[thm]{Lemma}
\newtheorem{exa}[thm]{Example}
\newtheorem{exas}[thm]{Examples}
\newtheorem{pro}[thm]{Proposition}

\def\ol{\overline}

\def\R{\mathbb{R}}

\def\s3{\mathfrak{s}^3}

\def\bi{\begin{enumerate}}
\def\ei{\end{enumerate}}

\newcommand{\grad}{{\rm grad}}

\newcommand{\Ric}{{\rm Ric}}
\newcommand{\rick}{{\rm ric}}

\newcommand{\sign}{{\rm sign}}
\newcommand{\bd} {\begin{displaymath}}
\newcommand{\ed} {\end{displaymath}}

\newcommand{\diver}{{\rm div}}

\title{Conformally Einstein product spaces}
\author[Wolfgang K\"uhnel \& Hans-Bert Rademacher]{Wolfgang 
K\"uhnel \& Hans-Bert Rademacher}
\begin{document}
%Version June 29, 2015 

\thanks{This work was done while the first author enjoyed
the hospitality of the Math.\ Dept.\ at the 
University of Leipzig. It was partially
supported by the DFG under the program SPP 1154.}
\subjclass[2010]{Primary: 53C25; Secondary: 53B30, 53C50, 53C80, 83C20.}
\keywords{conformal mapping, Einstein metric, warped product}
\begin{abstract}
We study pseudo-Riemannian Einstein manifolds
which are conformally equivalent with a metric product of two
pseudo-Riemannian manifolds. 
Particularly interesting is the case where one of these
manifolds is 1-dimensional and the case where the conformal factor
depends on both manifolds simultaneously.
If both factors are at least 3-dimensional then the 
latter case reduces to the product of
two Einstein spaces, each of the special type
admitting a non-trivial conformal gradient field.
These are completely classified.
If each factor is 2-dimensional, there is a special family
of examples of non-constant curvature
(called extremal metrics by Calabi), where in each factor
the gradient of the Gaussian curvature is a conformal vector field. 
Then the metric of the 2-manifold is a warped product where the warping function
is the 
%%%%%%%%%%%%%%%%%%%%%%%%%%%Change%%%%%%%%%%%%%%%%%%%%%%%%%%
first
%%%%%%%%%%%%%%%%%%%%%%%%%%%%%%%%%%%%%%%%%%%%%%%%%%%%%%%%%%%%%
derivative of the Gaussian curvature.
Moreover we find explicit examples of 
Einstein warped products with a 1-dimensional fibre and such
with a 2-dimensional base.
Therefore in the 4-dimensional case our Main Theorem 
points towards a local classification of conformally Einstein products.
Finally we prove an assertion in the book by A.Besse on complete
Einstein warped products with a 2-dimensional base. All solutions
can be explicitly written in terms of integrals of elementary functions.
\end{abstract}
\maketitle
\section{Introduction and notations}
\label{sec:basic}
%%%%%%%%%%%%%%%%%%%%%%%%%%%%%%%%%%%%%%%%%%%%%%%%%%%%%%%%%

\medskip
We consider a {\sf pseudo-Riemannian manifold} $(M,g),$
which is defined as a smooth $n$-manifold $M$ 
(here {\it smooth} means of class 
%%%%%%%%%%%%%%%%%%%%%Change
%$C^{\infty}$, at least 
$C^4$) together
with a pseudo-Riemannian metric of arbitrary signature 
$(j,n-j), 0 \le j \le n.$
All manifolds are asssumed to be connected.
A {\sf conformal mapping} between two pseudo-Riemannian manifolds 
$(M,g), (N,h)$ is a smooth mapping $F: (M,g) \rightarrow (N,h)$
with the property $F^*h =\varphi^{-2}\,g$ for a smooth
positive function $\varphi : M \rightarrow \R_+.$ 
In more detail this means that the equation
$$ h_{F(x)}\left(dF_x(X), dF_x(Y)\right)=\varphi^{-2}(x)g_x(X,Y)$$
holds for all tangent vectors $X,Y \in T_xM.$ Particular cases are
{\sf homotheties} for which 
$\varphi$ is constant, and {\sf isometries}
for which $\varphi=1.$

\medskip
A (local) one-parameter group $\Phi_t$ of conformal mappings of a
manifold into itself generates a 
{\sf conformal (Killing) vector field} 
$V$, sometimes also called
an {\sf infinitesimal conformal transformation,} 
with $V = \frac{\partial}{\partial t}\Phi_t$. 
We need to assume that $V$ itself is of class at least $C^3$. Conversely, 
any conformal vector field generates a local one-parameter group
of conformal mappings.
It is well known \cite{Ya57} that a vector field $V$ is 
conformal if and only if the {\sf Lie derivative}
${\mathcal L}_V g$ of the metric $g$ in direction of the vector field 
$V$ satisfies the equation 
\begin{equation}
\label{eq:nabla}
{\mathcal L}_V g= 2\sigma g
\end{equation} for a certain
smooth function $\sigma: M \rightarrow \R$.
Necessarily this conformal factor $\sigma$ coincides with the
divergence of $V$, up to a constant: 
$\sigma = \diver V/n\,.$ Particular cases of conformal vector fields are
{\sf homothetic vector fields} 
for which $\sigma$ is constant, and 
{\sf isometric vector fields},
also called {\sf Killing vector fields},
for which $\sigma =0.$  On a (pseudo-)Euclidean space the divergence
of a conformal vector field is always a linear function.
%This follow from Corollary~\ref{cor:ODE} below.

\medskip
Furthermore it is well known that
the image of a lightlike geodesic under any conformal mapping
is again a lightlike geodesic and that
for any lightlike geodesic $\gamma$ and any conformal
vector field $V$ the quantity $g(\gamma',V)$ is constant along 
$\gamma$.
Conformal vector fields $V$ with non-vanishing $g(V,V)$ 
can be made into Killing fields
within the same conformal class of metrics,
namely, for the metric
 $\ol{g}= |g(V,V)|^{-1} \,g\,.$
This is a special case of a so-called inessential conformal vector
field.

\medskip
A vector field $V$ on a pseudo-Riemannian manifold is called {\sf closed}
if it is locally a gradient field,
i.e., if locally there exists a function $f$ such that
$V=\grad f$. Consequently, from Equation~\ref{eq:nabla}
and $\; {\mathcal L}_Vg = 2\nabla ^2f \;$ 
we see that a closed vector field $V$ is conformal if and only
\begin{equation}
\label{eq:closed}
\nabla_X V=\sigma X 
\end{equation}
for all $X$ or, equivalently $\nabla^2f = \sigma g$. Here  
$\nabla^2 f (X,Y)=g\left(\nabla_X\grad f,Y\right)\,$
denotes the {\sf Hessian} $(0,2)$-tensor
and $n \sigma = \Delta f  = \diver \ (\grad \ f)\;$
is the {\sf Laplacian} of $f$.
If the symbol $\; ( \quad  )^{\circ}\;$
denotes the traceless part of a $\; (0,2)$-tensor,  
then $\; \grad \ f\;$ is conformal if and only if $\; (\nabla ^2 f)^{\circ}
\equiv 0\;$.
This equation 
\begin{equation}
\label{eq:hess}
(\nabla^2 f)^{\circ} = 0
\end{equation}
allows explicit solutions
in many cases, for Riemannian as well as for pseudo-Riemannian manifolds. 
%See the discussion in Section~\ref{sec:gradient} below.

\medskip

%A vector field is called {\sf complete} if its
%flow is globally defined as a 1-parameter group  
%$(\Phi_t)_{t \in \mathbb{R}}$ of diffeomorphisms, i.e.
%a global smooth mapping $\Phi \colon \R \times M \rightarrow M$, 
%$(t,x) \mapsto \Phi_t(x)$ satisfying 
%$\Phi_{t+s} = \Phi_t \circ \Phi_s$ and $\Phi_0 = {\rm Id}$.
As usual, 
\begin{equation}
\label{eq:R}
R(X,Y)Z = \nabla_X\nabla_YZ - \nabla_Y\nabla_XZ 
- \nabla_{[X,Y]}Z
\end{equation}
denotes the {\sf curvature $(1,3)$-tensor}.
Then the {\sf Ricci tensor}
as a symmetric $(0,2)$-tensor is defined
by the equation 
$$\Ric(X,Y)={\rm trace} \big(V \mapsto R(V,X)Y\big).$$
The associated $(1,1)$ tensor is denoted by 
$\rick$. Thus $\Ric(X,Y)=g\left(\rick(X),Y\right).$ Its trace
$S= {\rm trace} \big(V \mapsto \rick(V)\big)$
is called the {\sf scalar curvature.}
A manifold is {\sf conformally flat,}
if every point has a neighborhood
which is conformally equivalent to an open subset of 
a pseudo-Euclidean space.

\medskip
A pseudo-Riemannian manifold 
is called an {\sf Einstein space}
if the equation
\begin{equation}
\label{eq:Einstein}
\Ric = \lambda g
\end{equation}
holds with a factor $\lambda=S/n$ which is necessarily constant
if $n\geq 3$ and which is then called the {\sf Einstein constant}.
%%%%%%%%%%%%%%%%%%%%%Change%%%%%%%%%%%%%%%%%%%%%%
On a surface, i.e. for $n=2$ a pseudo-Riemannian metric is
called Einstein, if it has constant Gaussian curvature.
%%%%%%%%%%%%%%%%%%%%%%%%%%%%%%%%%%%%%%%%%%%%%%%%%%
For convenience the {\sf normalized Einstein constant} 
or {\sf normalized scalar curvature} 
will be denoted by $k = \lambda/(n-1)$ so that
we have $k=1$ on the unit sphere of any dimension.
For $n=2$ we have $\lambda = k = S/2 = K$ (Gaussian curvature).
For a survey on Einstein spaces in general we refer to
\cite{Be87}, \cite{De00}. A short version of this article is \cite{KR13}.
%%%%%%%%%%%%%%%%%%%%%%%%%%%%%%%%%%%%%%%%%%%%%%%%%%%%%%%%%%%%%%%%%%%
%%%%%%%%%%%%%%%%%%%%%%%%%%%%%%%%%%%%%%%%%%%%%%%%%%%%%%%%%%%%%%%%%%%
\section{Conformally  Einstein spaces: Basic equations}
\label{sec:Basic} 
We start with the very basic formula for the change of the Ricci
tensor under a conformal change of the metric.
This formula provides a way to classify all conformal
mappings between two Einstein spaces.
It is -- however -- more difficult to classify those Riemannian
metrics that are conformally Einstein, see 
\cite{Br23}, \cite{GN06}, \cite{GL09}, \cite{HPW}, \cite{MO}.
This is somehow in contrast with the case of conformally flat metrics.
It is clear that a 3-dimensional metric is conformally Einstein
if and only if it is conformally flat.
The problem becomes interesting in 
%%%%%%%%%%%%%%%%change%%%%%%%%%%%%%%%%%%%%%%%%%%%%%
dimensions $n \geq 4$.
Here it is well known that 
%%%%%%%%%%%%%%%%%%change%%%%%%%%%%%%%%%%%%%%%%%%%%%%%%%%%%%%%%%%
a harmonic Weyl tensor is a necessary condition
for a metric to be conformally Einstein. However, sufficient or equivalent
conditions are difficult to obtain, see \cite{Li01}, \cite{Li06}.
The case of conformally Einstein K\"ahler metrics was treated
in \cite{DM11},\cite{Masch}.
\begin{lem}
\label{lem:ric}
The following formula holds for any conformal change
$g \mapsto \overline{g} = \varphi^{-2}g$ of a metric on an
$n$-dimensional
manifold:
\begin{equation}
\label{eq:ric}
\overline{\Ric} - \Ric = \varphi^{-2} \Big( (n-2) \cdot \varphi \cdot 
\nabla^2 \varphi + \Big[\varphi \cdot \Delta \varphi - (n-1) \cdot 
\|\grad \varphi\|^2 \Big]\cdot g \Big).
\end{equation}

Consequently, the metric $\overline{g}$ is Einstein if and only if
the equation
\begin{equation}
\label{eq:conf-Einstein}
\varphi \cdot \Ric + (n-2) \cdot \nabla^2\varphi
=\theta \cdot g
\end{equation}
holds for some function $\theta$ or, equivalently,
$$\varphi \cdot (\Ric)^\circ + (n-2)\cdot (\nabla^2\varphi)^\circ =0$$
where $( \ )^\circ$ denotes the trace-free part.

\end{lem}
Equation \ref{eq:ric}
follows from the relationship between the two
Levi-Civita connections $\nabla,\overline{\nabla}$ associated 
with $g$ and $\overline{g}$:
$$\overline{\nabla}_XY - \nabla_XY = - X(\log \varphi ) Y - Y(\log \varphi ) X 
+ g(X,Y)\grad(\log \varphi).$$
Equation~\ref{eq:ric} is a standard formula,
see \cite[1.159]{Be87} or \cite[8.27]{Ku02}.
%Equation \ref{eq:lie-ric} can be found in \cite[p.160]{Ya57}.

\begin{cor}\label{conformalEinstein}
A metric $g$ on a manifold $M$ is (locally or globally) conformally Einstein
if and only if there is a (local or global) positive solution
$\varphi$ of the equation 
$\varphi \cdot (\Ric)^\circ + (n-2)\cdot (\nabla^2\varphi)^\circ =0.$
\end{cor}

%\begin{lem}
%A necessary condition for a manifold to be conformally Einstein is
%the integrability of the eigenspaces of the Ricci tensor.
%More precisely this holds in every part where the mutiplicities 
%of the eigenvalues are constant.
%\end{lem}

%Proof. 
%By the equation in Corollary~\ref{conformalEinstein} it is sufficient
%to show that the eigenspaces of the Hessian of $\varphi$ are
%integrable. This holds for any function by the following argument.
%If all multiplicities are one nothing has to be proved.
%Let us assume there is an open set with an eigenspace of
%higher multiplicity, and let $X,Y$ be two linearly independent
%vector fields in this eigen distribution, i.e. 
%$$\nabla_X\grad \varphi = \lambda X, \nabla_Y\grad \varphi = \lambda Y. $$
%We have to show that $\nabla_{[X,Y]}\grad \varphi = \lambda [X,Y]$.
%Let $A(X,Y) = \nabla_{[X,Y]}\grad \varphi - \lambda [X,Y]$
%denote the difference. Then we have the Ricci equation
%$$R(X,Y)\grad \varphi = \nabla_X(\lambda Y) - \nabla_Y (\lambda X)-
%\nabla_{[X,Y]}\grad \varphi = (X\lambda) Y - (Y\lambda) X - A(X,Y).$$
%Since the left hand side is a tensor field we can conclude
%that $A$ is a tensor field also, in particular
%$$A(fX,Y) = fA(X,Y)$$ for any scalar function $f$.
%But $$A(fX,Y) = \nabla_{[fX,Y]}\grad \varphi - \lambda [fX,Y]
%= fA(X,Y) - (Yf)(\nabla_X\grad \varphi - \lambda X).$$

\begin{exa}\rm (conformal cylinder, 
generalized Mercator projection)\label{Mercator}

Let $M_*$ be an $(n-1)$-dimensional Einstein space with
Einstein constant $\lambda_* = n-2$. Then the cylinder $M = \R \times M_*$
with the product metric $g = dt^2 + g_*$ is conformally Einstein:
The metric $\overline{g} = \cosh^{-2}t \cdot g$   
is Einstein with $\overline{\lambda} = n-1$.
If $M_*$ is the unit\linebreak $(n-1)$-sphere then $(M,g)$ is a cylinder
representing the (classical but $n$-dimensional) 
{\it Mercator projection} from the $n$-sphere without north and south pole. 
With $\varphi = \cosh t$ the equation in Corollary~\ref{conformalEinstein}
is satisfied by the block matrix structure
\[\Ric = \left( \begin{array}{cc}
0&0\\0&\Ric_* \end{array} \right) = \left( \begin{array}{cc}
0&0\\0&(n-2)g_* \end{array} \right), \quad \nabla^2 \varphi = \left( \begin{array}{cc}
\varphi''&0\\0&0\end{array} \right) = \left( \begin{array}{cc}
\varphi&0\\0&0\end{array} \right),\] 
\[(\Ric)^\circ = \left( \begin{array}{cc}
-\frac{(n-1)(n-2)}{n}&0\\0&\frac{n-2}{n}g_* \end{array} \right), \quad (\nabla^2 \varphi)^\circ 
=  \left( \begin{array}{cc} \frac{(n-1)\varphi}{n}&0\\0&-\frac{\varphi}{n}g_*\end{array} \right).\] 
In the special case of a compact Einstein space $M_*$ this generalized Mercator
projection is the result of \cite[Thm.2.1]{MO}, see
Corollary~\ref{Moroianu}.
The transition from the conformal cylinder 
$\overline{g} = \cosh^{-2}t\big(dt^2 + g_*\big)$ 
to the more familiar version $\overline{g} = dr^2 + \sin^2r \ g_*$ in parallel
coordinates around the equator with $r=\pi/2$ is achieved by
the parameter transformation $(-\infty,\infty) \ni t \mapsto r(t) \in (0,\pi)$ 
with $\sin r = \cosh^{-1} t$ and
$dr/dt = \cosh^{-1}t$ leading to the Gudermann function $r(t) = \int_{-\infty}^t
\cosh^{-1}\tau \ d\tau = 2 \arctan e^t$. 
However, $\overline{g}$ is not complete, as the classical Mercator
projection shows.

\medskip
Similarly, the metric $\overline{g} = \sin^{-2}t(dt^2 + g_*)$ is an Einstein
metric on the cylinder $\overline{M} = (0,\pi) \times M_*$
with $\lambda = -(n-1)$ if $(M_*,g_*)$ is Einstein
with $\lambda_* = -(n-2)$.
Moreover, if $(M_*,g_*)$ is the hyperbolic $(n-1)$-space then 
$(\overline{M},\overline{g})$ is isometric 
%%%%%%%%%%%%%%%% change with 
to the hyperbolic $n$-space
with the metric $\overline{g} = dr^ 2 + \cosh^2 r \ g_*, \ r \in \R$.
The transformation is given by $\sin t = \cosh^{-1} r$ and 
$t(r) = 2 \arctan e^r$.
One could call this a {\it hyperbolic Mercator projection}.
In particular, $\overline{g}$ is complete if $g_*$ is 
Riemannian and complete.

\end{exa}
\begin{cor} {\rm (Brinkmann \cite{Br25})}\label{Brinkmann}

If $g$ is an Einstein metric then $\overline{g}$ is also an Einstein
metric if and only if  $(\nabla^2\varphi)^\circ =0.$ 
\end{cor}
This follows directly from Equation \ref{eq:conf-Einstein}.
This case of conformal changes between two Einstein metrics was studied 
by many authors, starting with Brinkmann \cite{Br25}.
The equation $(\nabla^2\varphi)^\circ =0$ can be explicitly solved, 
compare \cite{KR09}.
Roughly the results are the following:
As long as $g(\grad\varphi,\grad \varphi) \not = 0$, the metric is a warped
product $g = \epsilon dt^2 + (\varphi'(t))^2g_*$ with an
$(n-1)$-dimensional
Einstein space $(M_*,g_*)$, $\epsilon = \pm 1$, and where $\varphi$ depends
only on $t$ and satifies the following differential equations:
\begin{equation}\label{eq:Brinkmann}
\varphi''' + \epsilon k \varphi' = 0, \quad (\varphi'')^2 + \epsilon k
(\varphi')^2 = \epsilon k_*
\end{equation}
By the second equation the normalized scalar curvature $k_*$ of $g_*$
appears as a constant of integration for
the first equation. We can integrate the first equation also as
$\varphi'' + \epsilon k\varphi = c$ but this constant $c$ is not essential:
We can add freely a constant to the function $\varphi$ itself without
changing the equation. 
However, $c$ becomes essential if $\varphi$ becomes a conformal factor.
If $g(\grad \varphi,\grad \varphi) = 0$ on an open subset then we have
$\nabla^2\varphi = 0$ and ${\rm Ric} = \overline{\rm Ric} = 0$, 
see \cite[Thm.3.12]{KR09}.
\begin{lem}
For a real constant $c$ and a positive function $\varphi$
we have the following equations for the function $\varphi^c$:
\begin{eqnarray}
\label{eq:c-power}
\grad \varphi^c &=& c \varphi^{c-1}\grad \varphi \\
\nabla^2 \varphi^c &=& c \varphi^{c-1}\nabla^2 \varphi +
c(c-1)\varphi^{c-2} d\varphi \otimes d\varphi
\end{eqnarray}
\end{lem}

{\it Proof.} The first equation is obvious from the chain rule
$d(\varphi^c) = c \varphi^{c-1}\cdot d\varphi$. 
For the second equation we calculate
\begin{eqnarray*}
\nabla_X \grad \varphi^c &=& \nabla_X (c \varphi^{c-1} \cdot \grad \varphi)\\
&=& c\varphi^{c-1}\cdot \nabla_X \grad
\varphi + cX(\varphi^{c-1}) \cdot \grad \varphi \\
&=&  c\varphi^{c-1}\cdot \nabla_X \grad
\varphi + c(c-1)\varphi^{c-2}\cdot d\varphi(X) \cdot \grad \varphi.
\end{eqnarray*}
\begin{cor} \label{quasi-Einstein}
If a function $\varphi > 0$ on an $n$-dimensional manifold 
$(M,g)$ satisfies the equation
$\varphi \cdot \Ric + (n-1) \cdot \nabla^2\varphi=0$
then with $c := \frac{n-1}{ n-2}$ the metric  $\overline{g} =
\varphi^{-2c}g$ 
satisfies the equation 
$$\frac{n-2}{n-1}\cdot \overline{\Ric} = \overline{\varphi}\cdot \overline{g}
+ \varphi^{-2}\cdot d\varphi \otimes d\varphi$$
for some scalar function $\overline{\varphi}$.
Therefore $\overline{g}$ is a quasi Einstein metric in the sense of \cite{CG}.
\end{cor}

%\pagebreak
\begin{exa}\rm
A quasi Einstein metric conformally equivalent with the
hyperbolic space.

We start with the hyperbolic space $(H^n,g_{-1})$ with the
metric in polar coordinates $g_{-1} = dr^2 + \sinh^2(r)g_1$
where $g_1$ is the metric on the unit $(n-1)$-sphere.
The function $\varphi(r) = \cosh(r)$ satisfies the equation
$\nabla^2\varphi = \varphi g_{-1}$, and $\Ric = -(n-1)g_{-1}$. 
Let $c := \frac{n-1}{n-2}$.
Then the metric
$$\overline{g} = \varphi^{-2c}g_{-1} =
\cosh^{-2(n-1)/(n-2)}(r)g_{-1}$$ 
is quasi Einstein 
by Corollary~\ref{quasi-Einstein} but not Einstein.
More precisely we have
$$\frac{n-2}{n-1}\cdot \overline{\Ric} = -n\tanh^2(r)\cdot g_{-1}
+ \cosh^{-2}(r)\cdot d\varphi \otimes d\varphi
%= -n\sinh^2(r)\cdot \overline{g}
%+ \cosh^{-2}\cdot d\varphi \otimes d\varphi
.$$
\end{exa}

\section{Conformally  Einstein products: Main Theorem}
\label{sec:Product} 
In the special case of an Einstein space that is conformally equivalent with
a Riemannian product of two manifolds several results
have been obtained, see \cite{Co00}, \cite{MO}, \cite{HPW},
\cite{Ta85}, \cite{TK}.
If the conformal factor depends only on one side
this is close to the case of an Einstein warped product which -- in general --
is not yet solved.
If the conformal factor depends on both sides of the product then
we can completely classify the solutions.
%We present a few examples at the end of this section,
%after a generalization of the main result in \cite{MO}.
Our Main Theorem~\ref{pseudoriem} studies precisely this case, thus 
generalizing and completing the main result of \cite{Cl08}.
Here the 4-dimensional case is special and leads to the class
of extremal metrics in the sense of E.Calabi \cite[18.4]{De00}.
We start with the simplest case of a function that depends only
on one real variable.
\begin{pro} \label{f(t)} {\rm (The type $\R \times M_*$ with
a 1-dimensional base)}

If $f$ is a non-constant function depending only on the real parameter $t$ then the
metric $\overline{g} = f^{-2}(\epsilon dt^2 + g_*)$ is Einstein 
if and only if $(M_*,g_*)$ is an $n$-dimensional Einstein space 
and $f$ satisfies the ODE \ $k_* f^2 - \epsilon (f')^2 = \overline{k}$.

\end{pro}

{\it Proof.}
We regard $f$ also as a function on the product with $M_*$.
Then for $g = \epsilon dt^2 + g_*$ we have the block matrix structure
\[\Ric = \left( \begin{array}{cc}
0&0\\0&\Ric_* \end{array} \right), \quad \nabla^2 f = \epsilon f''
(\epsilon dt^2) = \left( \begin{array}{cc}
\epsilon f''&0\\0&0\end{array} \right)\] 
and $\Delta f = \epsilon f'', \ ||\grad f||^2 = \epsilon (f')^2$. This implies
$$\overline{\Ric} = \Ric + (n-1)f^{-1}\nabla^2 f + 
\big[\epsilon f^{-1}f'' - \epsilon n f^{-2}(f')^2\big] \cdot (\epsilon dt^2 + g_*).$$
From the block matrix structure it follows that 
$\overline{g}$ is Einstein if and only if two
equations are satisfied:
The first one is the scalar equation
$$\overline{\lambda} = \epsilon n \big[ ff''- f'^2\big],$$
the second one is the tensor equation
$$ \overline{\lambda} g_* = f^2\Ric_* + \epsilon \big[ff''- n f'^2\big]g_*$$
where $\overline{\lambda}$ is the Einstein constant of $\overline{g}$.
A necessary condition is that $g_*$ is also Einstein with ${\rm Ric}_*
= \lambda_*g_*$ where $\lambda_*$ depends only on $t$ and
$\overline{\lambda}$. Therefore $\lambda_*$ is constant even if $n=2$.
In the Riemannian case $\epsilon = 1$ we obtain
$n(ff''-f'^2) = \overline{\lambda} = f^2\lambda_* + ff'' - n f'^2$ and therefore
$$f'' = \frac{\lambda_*}{n}f = k_*f.$$
Here the coefficients $k_* = \lambda_*/(n-1)$ and 
$\overline{k}=\overline{\lambda}/n$
are the normalized scalar curvatures of $g_*$ and $\overline{g}$, respectively.

\medskip
In the general case we get \ 
$\epsilon n \big[ ff''- f'^2\big] = \overline{\lambda} = 
f^2\lambda_* + \epsilon \big[ ff'' - n f'^2\big]$ and therefore
$$f'' = \epsilon k_*f.$$

We obtain the explicit relationship between $k_*$ and $\overline{k}$
by inserting the last equation
%%%%%%%%%%%%%%%%%change
into the preceding one
$$k_*f^2 - \epsilon f'^2 =  \overline{k}.$$
%By differentiating we obtain the equation
%$$f'' = \epsilon k_* f,$$ unless $f$ is constant,
%and consequently
%$$ff''-\epsilon f'^2 = \overline{k}.$$
We remark that an equation of the type $y'' = F(y,y')$
can be integrated in two steps, see \cite[\S 11.VI]{Wa}.
Here $k_*$ appears as one constant of integration.
Up to scaling we have to consider only the cases $\overline{k} =
1,0,-1$. The case $k_* = 0$ leads to linear solutions by $f'' = 0$.
Furthermore the cases of $k_* = \pm 1$
are particularly simple to handle. They lead to the standard equation
$(f')^2 \pm f^2 = 0$ or $(f')^2 \pm f^2 = \pm 1$
with all possible combinations of signs.

\medskip
In the Riemannian case $\epsilon = 1$ and with the initial 
conditions $f(0) = 1, f'(0) = 0$ we have the following solutions:

\smallskip
$f(t) = \cosh t$ if $\overline{k} = 1, k_* = 1$ (this
case corresponds to \cite[Thm.2.1]{MO}, compare
the Mercator projection in Example~\ref{Mercator}),

$f(t) = \cos t$ if $\overline{k} = -1, k_* = -1$ 
(compare Example~\ref{Mercator}),

$f(t) = 1$ if $\overline{k} = 0, k_* = 0$.

\medskip
With the initial conditions $f(0) = 1, f'(0) = 1$ we have
the solutions:

\smallskip
$f(t) = e^t$ if $\overline{k} = 0, k_* = 1$ (this case coincides
with the flat metric
$dr^2 + r^2 g_*$ in polar coordinates with $r = e^{-t}$ 
if $M_*$ is the unit sphere)

$f(t) = t+1$ if $\overline{k} = -1, k_* = 0$,

$f(t) = \frac{1}{2\sqrt{2}}\big[(\sqrt{2}+1)e^{\sqrt{2}\ t}
+ (\sqrt{2}-1)e^{-\sqrt{2}\ t}\big]$ if
$\overline{k} =  1, k_* = 2$.

In the last case the initial condition implies $k_* = 2$
leading to the equation $f'' = 2f$. This explains the factor $\sqrt{2}$.

\medskip
With the initial conditions $f(0) = 0, f'(0) = 1$ we have
the solutions:

$f(t) = \sinh t$ if $\overline{k} = -1, k_* = 1$,

$f(t) = \sin t$ if $\overline{k} = -1, k_* = -1$ 
(compare Example~\ref{Mercator}),

$f(t) = t$ if $\overline{k} = -1, k_* = 0$ (this is known
as the Poincar\'{e} halfspace model for the hyperbolic space if $g_*$
is flat).

\medskip
All these solutions define admissible metrics only outside the zeros of $f$.
The missing solution with $f(0) = f'(0) = 0$ is trivial
and occurs only for $\overline{k} = 0.$ 
It does not lead to a metric $\overline{g}$. \hfill $\Box$

\medskip
{\sc Remark:} In principle this classification is similar to that of
Einstein metrics as warped products $dt^2 + \varphi'^2(t)g_*$
in Corollary~\ref{Brinkmann}.
The differential equations are similar, and so are the solutions.
Nevertheless, from the global point of view, there are
essential differences.

\bigskip
In more generality we have the following theorem on products
of pseudo-Riemannian manifolds of arbitrary dimensions.
In the Riemannian case it is stated in \cite{Cl08}
but the 4-dimensional case of a product of two surfaces with
nonconstant curvature is missing there.
Therefore Theorem 1 in \cite{Cl08} is not literally true in dimension four.
\begin{thm} {\rm (Main Theorem on conformally Einstein products)}
\label{pseudoriem}

Let $(M^n,\widetilde{g})$ and $(M_*^{n_*},g_*)$ be pseudo-Riemannian
manifolds
with $n +n_* \geq 3$.
If $f(y,x)$ is a non-constant function depending on $y\in M$ 
and $x \in M_*$ 
and if the metric $\overline{g} = f^{-2}(\widetilde{g} + g_*)$ on $M \times M_*$
is Einstein 
then one of the following cases occurs:

\begin{enumerate}
\item $\overline{g}$ is a warped product, i.e., $f$ depends only on 
one of the factors $M$ or $M_*$. Moreover the fibre is an Einstein space.
%$t$ (the case discussed in Proposition~\ref{f(t)}) 
%\item $f$ depends only on $x$ (the case discussed in Corollary~\ref{conf-produc%t})
\item $f(y,x) = a(y) + b(x)$ with non-constant $a$ and non-constant $b$,
and both $(M,g)$ and $(M_*,g_*)$ are Einstein spaces
with normalized scalar curvatures $\widetilde{k},k_*$, 
and $a$ satisfies the equation 
$\widetilde{\nabla}^2a = \frac{\widetilde{\Delta} a}{n}\widetilde{g}$
%$\widetilde{\nabla}^2a = (-\widetilde{k}a}+c)g$
and, simultaneously, 
$b$ satisfies the equation 
${\nabla}_*^2b = \frac{{\Delta}_* b}{n_*}g_*$.

\noindent
If $n\geq 3$ or $n_*\geq 3$ then we have necessarily
$\widetilde{\nabla}^2a = (-\widetilde{k}a+c)g$
and, simultaneously, $\nabla_*^2b = (-k_*b + c)g_*$
with a constant $c$ and with $\widetilde{k} = -k_*$.
Such Einstein spaces can be (locally and globally) classified
\cite{KR09}.

\noindent
If $n=n_*=2$ then either the Gaussian curvatures are constant
and satisfy $\widetilde{K} = -K_*$, or they satisfy the equations
$\widetilde\nabla^2K = \frac{\Delta \widetilde K}{2}\widetilde{g}$
and $\nabla_*^2K_* = \frac{\Delta K_*}{2}g_*$. Such metrics
are called extremal in \cite[18.4]{De00}.
\end{enumerate}
Conversely, any Einstein warped product in $(1)$ is conformally equivalent
with a product space, and any two Einstein metrics $\widetilde{g},g_*$
with constant $\widetilde{k} = -k_*$ and with solutions $a(y),b(x)$
of the equations 
$\widetilde{\nabla}^2a = (-\widetilde{k}a+c)g$ and
$\nabla_*^2b = (-k_*b + c)g_*$ lead to an Einstein metric
$\overline{g} = (a+b)^{-2}(\widetilde{g}+g_*)$ on $M\times M_*$
in $(2)$.

\smallskip
If $n=n_*=2$ then there are also examples $M\times M_*$ with two surfaces
$M,M_*$ that are not of constant curvature, 
see Example~\ref{examplesurfaces} for the details.
However, by Corollary~\ref{compact}
%a theorem of Calabi \cite[Thm.18.14]{De00}
there are no compact examples that are (locally) conformally Einstein.
\end{thm}
{\sc Remark:}
A complete classification of Einstein warped products
is not known, compare \cite{Be87}, \cite{HPW}.
However, Einstein warped products with a 1-dimensional base are
easy to classify by the equations in Corollary~\ref{Brinkmann}, 
see \cite{Ku88}, \cite{KR97}, \cite{KR09}.
For the case of a 2-dimensional base see \cite[Thm.9.119]{Be87}.
The similar equations $\widetilde{\nabla}^2f = -\widetilde{k}f\widetilde{g}$ and
$\nabla_*^2f = \widetilde{k}fg_*$ hold for the divergence $f$
of a non-isometric conformal vector field on a complete non-flat
Riemannian product $M \times M_*$, see \cite[Thm.5]{Ta65}, \cite{TM67}.
The crucial condition $\widetilde{k} = -k_*$ 
occurs also in the Fefferman-Graham
ambient metric construction on $M\times M_* \times [0,\varepsilon)$,
see \cite{GL09}.

\bigskip
{\it Proof.}
Let $g = \widetilde{g} + g_*$ and $N = n+n_*-1 \geq 2$. 
Then by Equation~\ref{eq:ric}
$$f^2\big(\overline{\Ric} - \Ric\big) = (N-1) f \cdot 
\nabla^2 f + \Big[f \cdot \Delta f - N\cdot 
\|\grad f\|^2 \Big]\cdot g \,.$$ 
%%%%%%%%%%%%%%%%change%%%%%%%%%%%%%%%%%%%%%%%%%%%%%

The Einstein condition for $\overline{g}$ 
implies that $\nabla^2 f$ admits an orthogonal decomposition
into some tensor on $M$ and another tensor on $M_*$.
%has the form
%\[ \nabla^2 f = \left( \begin{array}{cc}
%\frac{\partial^2f}{\partial t^2}&0\\0&*\end{array} \right)\] 
%with some $(0,2)$-tensor on $M_*$ denoted by $*$.
In coordinates $y_1, \ldots y_n, x_1, \ldots, x_{n_*}$ on $M \times
M_*$ this implies
$\frac{\partial^2f}{dy_jdx_i} = 0$ for any $i,j$.
Therefore $f$ splits as
$$f(y,x) = a(y) + b(x)$$ 
with functions $a$ on $M$ and $b$ of $M_*$.
This implies
\[ \nabla^2 f = \left( \begin{array}{cc}
\widetilde{\nabla}^2a &0\\0&\nabla_*^2 b\end{array} \right)\] 
and 
$$||\grad f||^2 =  ||\widetilde{\grad} a||^2 + ||\grad_* b||^2, 
\quad  \Delta f = \widetilde{\Delta} a + \Delta_* b.$$

\medskip
Therefore with $\overline{\Ric} = \overline{\lambda}\overline{g}$
Equation~\ref{eq:ric} can be written in block matrix form as
%\[
\begin{equation} \label{sternstern}
\overline{\lambda}\left( \begin{array}{cc}
\widetilde{g}&0\\0&g_*\end{array} \right) -  f^2\left( \begin{array}{cc}
\widetilde{\Ric} &0\\0&\Ric_*\end{array} \right)
=  (N-1)f\left( \begin{array}{cc}
\widetilde{\nabla}^2a &0\\0&\nabla_*^2 b\end{array} \right)
 + \Big[f \cdot \Delta f - N \cdot 
\|\grad f\|^2 \Big]\cdot  \left( \begin{array}{cc}
\widetilde{g}&0\\0&g_*\end{array} \right).
\end{equation}
%\]
From this equation it is obvious
that a constant function $a$ implies that $\widetilde{g}$ is Einstein 
and a constant function $b$
implies that $g_*$ is Einstein.
In each of these cases $\overline{g}$ is a warped product metric
with an Einstein fibre. This is case (1) in the theorem. 

\medskip
If $a$ and $b$ are both non-constant we can pick tangent vectors
$Y$ on $M$ and $X$ on $M_*$ such that $\nabla_Ya \not = 0$ 
and $\nabla_Xb \not = 0$. 
Then we consider the covariant derivatives $\nabla_X$ and $\nabla_Y$
of the last equation    
with respect to $g$. By the product decomposition we have
$\nabla_X\widetilde{g} = \nabla_Yg_* = 0$ and $\nabla_Xa = \nabla_Yb =
0$.
If $N\geq 2$ then $\overline{\lambda}$ is constant.
Using this and $f = a+b$ the results are 
\[0 = 2f \nabla_Xb \cdot \widetilde{\Ric}
+ (N-1)\nabla_Xb\cdot \widetilde{\nabla}^2a 
 + \Big[\nabla_Xb \cdot \Delta f + f\nabla_X\Delta_*b- N\cdot 
\nabla_X\|\grad f\|^2 \Big]\cdot  
\widetilde{g},\]
\[0 = 2f \nabla_Ya \cdot {\Ric_*}
+ (N-1)\nabla_Ya\cdot {\nabla}_*^2b 
 + \Big[\nabla_Ya \cdot \Delta f + f\nabla_Y\widetilde{\Delta} a - N\cdot 
\nabla_Y\|\grad f\|^2 \Big]\cdot g_*\]
and, consequently,
\[0 = 2f \cdot \widetilde{\Ric}
+ (N-1)\cdot \widetilde{\nabla}^2a 
 + \Big[\Delta f + (\nabla_Xb)^{-1}\Big(f\nabla_X\Delta_*b- N\cdot 
\nabla_X\|\grad f\|^2\Big) \Big]\cdot  
\widetilde{g},\]
\[0 = 2f \cdot {\Ric_*}
+ (N-1)\cdot {\nabla}_*^2b 
 + \Big[\Delta f + (\nabla_Ya)^{-1}\Big(f\nabla_Y\widetilde{\Delta} a - N\cdot 
\nabla_Y\|\grad f\|^2\Big) \Big]\cdot g_*.\]
Differentiating once more and using $\nabla_X\widetilde{\nabla}^2a = 
\nabla_Y{\nabla}_*^2b = 0$ we obtain
\[0 = 2\nabla_Xf \cdot  \widetilde{\Ric} 
 + \nabla_X\Big[\Delta f + (\nabla_Xb)^{-1}\Big(f\nabla_X\Delta_*b- N\cdot 
\nabla_X\|\grad f\|^2\Big) \Big]\cdot  
\widetilde{g},\]
\[0 = 2\nabla_Yf \cdot  {\Ric_*} 
 + \nabla_Y\Big[\Delta f + (\nabla_Ya)^{-1}\Big(f\nabla_Y\widetilde{\Delta}a- N\cdot 
\nabla_Y\|\grad f\|^2\Big) \Big]\cdot g_*.\]
This implies that $\widetilde{g}$ and $g_*$ are Einstein metrics.
From the previous equations we get that in addition
$\widetilde{\nabla}^2a$ is a scalar multiple of $\widetilde{g}$
and that $\nabla_*^2b$ is a scalar multiple of $g_*$.
This is precisely the equation in Corollary~\ref{Brinkmann}.
On an Einstein space this equation can be completely and
explicitly solved, see Corollary~\ref{Brinkmann} and, in more detail,
\cite{Be87}, \cite{Ta65}, \cite{Ku88}, \cite{KR09}.
Moreover it follows that we have Einstein constants 
$\widetilde{\lambda}$ and $\lambda_*$ whenever $n \geq 3$ and $n_* \geq 3$.
If $n=2$ or $n_* = 2$ no conclusion about the constancy of the
curvature is possible.
However, if the gradient of $a$ or $b$ is isotropic, then it follows
$\widetilde{\Delta} a = 0$ and $\widetilde{k}=0$ or $\Delta_* b= 0$
and $k_* = 0$,
respectively. So in the sequel we can assume that the gradients
are either spacelike or timelike with $\epsilon_1 = \sign (||\grad a||^2), \epsilon_2 =
\sign(||\grad b||^2)$.

\medskip
Now let $n\geq 3$ and $n_*\geq 3$.
Then it is well known \cite{Ku88}, \cite{KR09} that 
$\widetilde{\nabla}^2a = (-\widetilde{k}a+c)\widetilde{g}$
and simultaneously $\nabla_*^2b = (-k_*b + c_*)g_*$
with constants $c,c_*$. 
Recall $\widetilde{\Delta} a = \epsilon_1na'', \Delta_*b =
\epsilon_2n_*b''$ and $a'''+\epsilon_2\widetilde{k}a' = 0 =
b'''+\epsilon_2k_*b'$
from Corollary~\ref{Brinkmann}.
We still have to prove $\widetilde{k} = -k_*$
and $c=c_*$.
Since $\widetilde{\rm Ric} = (n-1)\widetilde{k}\widetilde{g}$ and 
${\rm Ric}_* = (n_*-1)k_*g_*$ we consider Equation~\ref{sternstern}
and observe that the trace-free part has to vanish. This leads to the equation
$$(a+b)\big((n-1)\widetilde{k} - (n_*-1)k_*\big) +
(n+n_*-2)(-\widetilde{k}a+c+k_*b-c_*) = 0.$$
Since $a$ depends only on $y \in M$ and $b$ depends only
on $x \in M_*$ we conclude that $k_* = -\widetilde{k}$ and, by using
this equation, $c=c_*$.
The same is true if $n=2$ and $n_* \geq 3$ because a constant
scalar curvature on one of the factors implies that the scalar
curvature on the other factor is also constant.

\medskip
If $n=1$ or $n_* = 1$ the same holds, see Corollary~\ref{f(t,x)} below.

\medskip
If $n=n_*=2$ then we have $\nabla^2a = \epsilon_1a''\widetilde{g}$ and 
$\nabla_*^2b = \epsilon_2b''g_*$ where $(\ )'$ denotes differentiation 
%by
%%%%%%%%%%%%%%%%%%%%%%%%%%%%%%%%%%% changes
with respect to the arc length 
parameter $t$ resp.\ $s$ on the trajectories 
of the gradients of $a,b$, compare Corollary~\ref{Brinkmann}. 
The metrics are
$\widetilde{g} = \epsilon_1dt^2 \pm a'^2dx^2, g_* = \epsilon_2ds^2 \pm b'^2dy^2$.
All geometric quantities depend only on $a',b'$, not on an additive
constant
of $a$ or $b$, e.g.\ we have the Gaussian curvatures
$K = -\epsilon_1a'''/a', K_* = -\epsilon_2b'''/b'$.
Then Equation~\ref{sternstern} reads as follows:
$$\overline{\lambda} = (a+b)^2K + 2(a+b)\epsilon_1a'' + (a+b)(2\epsilon_1a''+2\epsilon_2b'') - 3(\epsilon_1a'^2+\epsilon_2b'^2)$$ 
$$\overline{\lambda} = (a+b)^2K_* + 2(a+b)\epsilon_2b'' + (a+b)(2\epsilon_1a''+2\epsilon_2b'') -
3(\epsilon_1a'^2+\epsilon_2b'^2).$$
In particular we have
\begin{equation}\label{eq:2x2}
0 = (a+b)(K-K_*)+ 2(\epsilon_1a''-\epsilon_2b'') = aK +2\epsilon_1a'' - (bK_* + 2\epsilon_2b'') + bK-aK_*
\end{equation}
where the first summand depends on $t$ only, the second one on $s$
only, and the third depends on both simultaneously.
By differentiating 
%by 
%%%%%%%%%%%%%%%%%%%%%%%%%%%%%changes
%%%%%%%%%%%%%%%%%%%%%%%%%%%%%%%%%%%%%
$t$ and $s$ we see that $b'K'-a'K_*' = 0$, 
hence $K'/a' = K_*'/b' = c$ is constant. Only the case $c \not = 0$ is
interesting here.
%Up to scaling we can assume that this constant is 1, i.e. $a=K, b=K_*$.
By adding some constant to $a$ and $b$ we can assume that $K = ca, K_*= cb$.
Therefore from Equation~\ref{eq:2x2} we obtain
%Since $K = -a'''/a', K_* = -b'''/b'$ is a standard formula for
%rotationally symmetric metrics, we obtain the eqations
$ca^2+2\epsilon_1a'' = d$ and $cb^2+2\epsilon_2b'' = d$ with a constant $d$.
These equations have non-constant solutions, see
Example~\ref{examplesurfaces} (4).

\medskip
Conversely, if $k_* = -\widetilde{k}$ is constant 
and if $a,b$ satisfy $\widetilde{\nabla}^2a = (-\widetilde{k}a+c)g$
and $\nabla_*^2b = (-k_*b + c)g_*$
%$a''+\epsilon_1\widetilde{k}a = c =b''+\epsilon_2k_*b$ 
then we see that Equation~\ref{sternstern} is satisfied because the constant coefficients of $a$ 
and $b$ vanish. This part holds also if $n=2$, $n_* \geq 3$ because
in this case the equations imply that the 2-manifold is of
constant curvature. It also holds if $n=n_*=2$ and if
the two curvatures are constant.
The Einstein constant $\overline{\lambda}$ depends on the
constant $\widetilde{k}=-k_*$ and the constants of integration
$c,c_1,c_2,d_1,d_2$ and $\epsilon_1,\epsilon_2$ from
the various differential equations on $M,M_*$ according to
Equation~\ref{eq:Brinkmann}:
$$a''+\epsilon_1\widetilde{k}a=\epsilon_1c, \ (a'')^2+\epsilon_1\widetilde{k}(a')^2 = c_1, \  
\ \epsilon_1(a')^2 +\widetilde{k}a^2 - 2ac = d_1,$$
$$b''+\epsilon_2k_*b=\epsilon_2c, \ (b'')^2+\epsilon_2k_*(b')^2 = c_2, \ 
\epsilon_2(b')^2 +k_*b^2-2bc = d_2$$
with the coupling $c_1 = \widetilde{k}d_1 + c^2, \ c_2= k_*d_2 + c^2$.
From Equation~\ref{sternstern} we obtain
\begin{eqnarray*}
\overline{\lambda} &=& (a+b)\Big[(n-1)\widetilde{k}(a+b) +
(N-1)(-\widetilde{k}a+c) + n(-\widetilde{k}a+c) + n_*(-k_*b+c) \Big]\\
&& -N(\epsilon_1a'^2+\epsilon_2b'^2)\\
&=& N\big[ -\epsilon_1a'^2-\widetilde{k}a^2 - \epsilon_2b'^2 -k_*b^2+ 2ac + 2bc \big]\\
&=& N(-d_1-d_2).
\end{eqnarray*}
Equivalently we get
%The result is the equation
%$$\overline{\lambda} = -N(d_1 + d_2)  \ \mbox{ or, equivalently, } \
$\overline{k} = -(d_1+d_2).$
Moreover, in the special case of dimensions $n=n_*=2$
there are in fact examples of non-constant curvature and with
a non-vanishing Einstein constant $\overline{\lambda}$, see 
Example~\ref{examplesurfaces} (4) below.
\hfill $\Box$
%The trace of the two last equations above leads to
%\[0 = 2\nabla_Xf \cdot  \widetilde{n}\widetilde{\lambda} 
% + \widetilde{n}\nabla_X\Big[\Delta f + (\nabla_Xb)^{-1}\Big(f\nabla_X\Delta_*b%- N\cdot 
%\nabla_X\|\grad f\|^2\Big) \Big],\]
%\[0 = 2\nabla_Yf \cdot  n_*{\lambda_*} 
% + n_*\nabla_Y\Big[\Delta f + (\nabla_Ya)^{-1}\Big(f\nabla_Y\widetilde{\Delta}a%- N\cdot 
%\nabla_Y\|\grad f\|^2\Big) \Big].\]

\begin{exas}\label{examplesurfaces}\rm

(1) On a Riemannian product the special solution $a(t) = \cos t, \ b(s) = \cosh s$ 
in the situation of Theorem~\ref{pseudoriem}
realizes the case $c=0$, \ $\widetilde{k} = c_1 = c_2 = d_1 = -d_2 = 1$
and $\overline{\lambda} = 0$.
The same $\overline{\lambda}$ is obtained for $a(t) = \cos t + c, \
b(s)=\cosh s -c$ with arbitrary $c$.

\medskip
The special solution $a(t) = \cos t, \ b(s) = e^s$ 
realizes the case $c=0$, \ $\widetilde{k} = c_1 = d_1 = 1$, $c_2 = d_2
= 0$ and $\overline{k} = -1$.
With $a(t) = \cos t, \ b(s) = \sinh s$ 
we have the case $c=0$, \ $\widetilde{k} = c_1 = d_1 = d_2 = 1$, $c_2 
= -1$ and $\overline{k} = -2$.

\medskip
A positive $\overline{k}$ is obtained in the case $\widetilde{k} = k_*
= 0$ and $a(t) = t^2+1, \ b(s) = s^2+1$ and, consequently, $c=2$ and
$d_1 = d_2 = -4$.
This leads to $\overline{k} = 8$. Unfortunately, no complete $M$ or $M_*$
admits this solution except the Euclidean space itself, compare
Corollary~\ref{Cleyton}.

\medskip
(2) If $n \geq 3, n_* \geq 3$ then special solutions $a,b$ satisfy
$$\widetilde{\nabla}^2a = -\widetilde{k}a\widetilde{g}, 
\quad \nabla_*^2b = -k_*bg_*.$$ 
If in addition the Einstein constants
$\widetilde{\lambda}
= (n-1)\widetilde{k}$ and $\lambda_* = (n_*-1)k_*$ are equal,
then $g$ and $\overline{g}$ are both Einstein metrics on $M \times
M_*$.
Therefore the conformal transformation $g \mapsto \overline{g} =
(a+b)^{-2}g$ 
satisfies the equation in Corollary~\ref{Brinkmann}. 
This can actually happen if $\widetilde{k} = k_* = 0$. 
%and it shows that the Cases (1) and (2) are not disjoint
%since a conformal transformation between Riemannian
%Einstein spaces can occur only for warped product \cite{KR09}.
A concrete example with a Ricci flat $M \times M_*$
can be constructed as in Case (3)
in the Examples~\ref{ppexamples}.

\medskip
(3) (a compact example)

Let $(M,g)= (S^n,g)$ be the unit sphere with normalized scalar curvature
$k=1$.
Take a second copy $(M_*,g_*) = (S^n,-g)$ with a negative definite metric
and with normalized scalar curvature $k_* =-1$ and let $c=0$.
Then on $M$ the function $a(t) = \cos t$ satifies the equations
$\nabla^2 a = -kag = -ag$ and $a'^2+a^2=1$, on $M_*$ the function 
$b(s) = \cos s $
satisfies the equations $\nabla_*b = -k_*bg_* = bg_*$ and $-b'^2-b^2 = -1$.
By the results of Theorem~\ref{pseudoriem}
the metric $\overline{g} = (a+b)^{-2}(g + g_*)$ defines an Einstein
space whenever $a+b \not = 0$. The Einstein constant is
$\overline{\lambda} = -2n(d_1+d_2)=0$. 
%$$\overline{\lambda} = (n-1)(a+b)^2 + 
%(2n-2)(a+b)(-a) + (a+b)n(b-a)-(2n-1)(a'^2-b'^2)
%$$
%$$= (a+b)(-(2n-1)(b-a)) -(2n-1)(a'^2-b'^2) = (2n-1)(-(a^2+a'^2) +
%(b^2+b'^2)) = 0.$$
Unfortunately the function $f(t,s)=a(t)+b(s)$ has zeros. This means
that the product metric on $M\times M_*$ ist locally but not globally conformally
Einstein. It would not
help to introduce an additive constant $c$ here because then we would have
$a(t)=\cos t +c$ and $b(s)=\cos s -c$.

\medskip
(4) (Wong \cite[Thm.10.1]{Wo43}, compare \cite[Sect.18]{De00}, \cite{DS00}, \cite[Sect.6]{KR08})

If $n = n_* = 2$ then the proof of Theorem~\ref{pseudoriem}
indicates that the Gaussian curvatures $K,K_*$ should satisfy
the equation $K^2+2K'' = c = K_*^2+2K_*''$.
Therefore, following Corollary~\ref{Brinkmann}, we consider the warped product
metric $\widetilde{g} = dt^2 + K'^2(t)dx^2$
admitting the solution 
$$\nabla^2K = \frac{\Delta K}{2}\widetilde{g} = K''\widetilde{g}$$
with the Gaussian curvature $-K'''/K'$. Since the ODE 
$$K'''+KK'= 0$$ holds by differentiating $2K'' + K^2 = c$, 
$K(t)$ itself becomes the Gaussian curvature of $g$, and 
$K\widetilde{g}$ becomes the Ricci tensor. Conversely,
by integration we obtain the constant $c$ satisfying $2K''+K^2 = c$.
Similarly we have the metric $g_* = ds^2 +K_*'^2(s)dy^2$ with the ODE
$$K_*'''+K_*K_*'= 0.$$ 
There are nonconstant solutions $K(t), K_*(s)$.
Particular cases are
$$K(t) = -12 t^{-2}, K_*(s) = -12 s^{-2}$$
with $c=0$.
Then for any choice of $c$ we can verify the equations in the proof of
Theorem~\ref{pseudoriem} for the function $f(t,s) = K(t)+K_*(s) = a(t)
+ b(s)$ (
%nothing but
%%%%%%%%%%%%%%%%%%%changes
which is nothing else than %%%%%%%%%%%%%changes
the scalar curvature of the product metric)
and for $X = \partial_s, Y = \partial_t$:
We have 
$${\nabla_*}_Xb = K_*', \widetilde{\nabla}_Y a = K', {\Delta}_*b = 2K_*'',
\widetilde{\Delta} a = 2 K'', ||\grad f||^2 = K'^2 + K_*'^2.$$ This implies
$$\widetilde{\rm Ric} = K\widetilde{g}, {\rm Ric}_* = K_* g_*,
2\widetilde{\nabla}^2 K = 2K'' \widetilde{g} = (c-K^2) \widetilde{g},  
2\nabla_*^2 K_* = 2K_*'' g_* = (c-K_*^2) g_*.$$
By $ 2K''K' = cK' - K^2K'$ a second integration step leads to the equation\linebreak
$K'^2 = cK - \frac{1}{3}K^3 + d$ with a constant $d$, 
similarly $K_*'^2 = c_*K_* - \frac{1}{3}K_*^3 + d_*$.\linebreak
Finally, 
%the 
%%%%%%%%%%%%%%%%%%%%%%%%changes
Equation~\ref{sternstern} is satisfied with the Einstein constant  
$$\overline{\lambda}= -(K^3+K_*^3 + 3K'^2 + 3K_*'^2)+3c(K+K_*) = -3(d+d_*):$$
%since 
%$$(K^3+3K'^2)' = 3K'(K^2+2K'') = 3cK', (K_*^3+3K_*'^2)'= 3K_*'(K_*^2+2K_*'') = %3cK_*':$$
\begin{eqnarray*}
&&(K+K_*)^2 K + 2(K+K_*)K'' + [2(K+K_*)(K''+K_*'') - 3(K'^2 + K_*^2)]\\
&=& (K+K_*)^2 K + (K+K_*)(c-K^2) + [(K+K_*)((2c-K^2-K_*^2) - 3(K'^2 + K_*^2)]\\
&=& (K+K_*)[K^2+KK_*+3c-2K^2-K_*^2] -3(K'^2+K_*'^2)\\
&=& 3c(K+K_*) - (K^3+K_*^3 +3K'^2+3K_*'^2)\\
&=&\overline{\lambda}.
\end{eqnarray*}
Similarly we have
$$(K+K_*)^2 K_* + 2(K+K_*)K_*'' + [2(K+K_*)(K''+K_*'') - 3(K'^2 + K_*^2)] = \overline{\lambda}.$$
This implies that the metric $\overline{g} = (K+K_*)^{-2}(dt^2 + K'^2
dx^2 + ds^2 + K_*'^2dy^2)$ is a 4-dimensional Einstein metric.
Up to scaling, these examples are unique by the formulas above.
Up to scaling, the particular case with $c=d=d_*=0$ is the Ricci flat metric
with $K(t) = -12t^{-2}, K_*(s) = -12 s^{-2}$ on an open part of $\R^4$ 
$$ \overline{g} = \frac{t^4s^4}{(t^2+s^2)^2}\Big( 
  dt^2+\frac{576}{t^6}dx^2
+ds^2+\frac{576}{s^6}dy^2\Big).$$ 
A picture of a surface of revolution in 3-space 
with the metric $dt^2 + \frac{576}{t^6}d\varphi^2$
is similar to that of Beltrami's surface (pseudosphere):
If the distance from the axis of revolution 
is $r(t) = 24 t^{-3}$ then the height along the axis is
$h(t) = \int_{t_0}^t \sqrt{1-(72\tau^{-4})^2}d\tau$ with $t \geq t_0$
where $t_0= \sqrt[4]{72}$
is the zero of $h' = \sqrt{1-r'^2}$. The $t_0$-circle is a singularity in space
with $K(t_0) = -12 t_0^{-2} = -\sqrt{2}$ and with 
one vanishing principal curvature, and the curvature $K(t) = -12t^{-2}$ tends
to zero from below for $t \to \infty$
with asymptotic principal curvatures $24^{-1}t^3 \approx 1/r$ 
and $-288 t^{-5} \approx Kr$. 
The surface is complete 
in the positive $t$-direction.
It is not obvious that the product of this surface with itself
is conformally Ricci flat.
\end{exas}

\begin{cor} {\rm (The compact case)}\label{compact}
\label{cor:compact}

%The only compact products $(M,\widetilde{g}) \times (M_*,g_*)$ 
%of pseudo-Riemannian Einstein spaces
%such that $\overline{g} = (f(y,x))^{-2}(\widetilde{g} + g_*)$ 
%is Einstein
%with a conformal factor $f(y,x) = a(y)+b(x)$ and nonconstant $a,b$
%are of the type of the example above, namely, a product of
%a Riemannian sphere with an anti-Riemannian sphere.
%The function $f$ is globally defined but has zeros.
%Therefore $\overline{g}$ does not define a compact Einstein space.

%\medskip ANDERE VERSION:
Let $(M,g) \times (M_*,g_*)$ be compact and assume
it is a (locally) conformally Einstein space
such that the conformal factor $f$ always depends on both $M$ and $M_*$
simultaneously. Then $M,M_*$ are round spheres, one
with a positive definite metric, the other one with a negativ definite
metric and such that the normalized scalar curvatures satisfy $k_* = -k$.  
However, there is no conformal factor $f$ without a zero.
Consequently, the metric $\overline{g} = f^{-2}(g+ g_*)$
does not define (globally) an Einstein metric on a compact manifold. 
\end{cor}

{\it Proof.}
By Theorem~\ref{pseudoriem} the function $f = a + b$ would 
constitute a nonconstant solution of the differential equations
$$\widetilde\nabla^2a = \frac{\Delta a}{n}g
\ \mbox{ or } \ {\nabla}_*^2b = \frac{\Delta_* b}{n_*}g_*,$$ respectively.
It is well known that topologically the only such compact manifold $M$ or
$M_*$ is a sphere, and that the metric must be positive or negative definite. 
This holds even if the functions $a$ or $b$ are only locally defined.
In a neighborhood of every point there is such a solution.
However, we note that passing to the negative of a metric does not
change the covariant derivative, does not change the Ricci tensor
or the Hessian $(0,2)$-tensor,
but it does change the sign of the scalar curvature.
In other words: If  ${\nabla}_*^2b = -k_*bg_*$
holds then for the metric $g_{**} = -g_*$ we have
${\nabla}_{**}^2b = -k_{**}bg_{**} = k_*bg_{**}. $
For $n \geq 3, n_* \geq 3$  
the Einstein condition then implies that both $(M,\widetilde{g})$
and $(M_*,g_*)$ are round spheres. 
The condition $\widetilde{k} = -k_*$ implies that one is Riemannian,
the other one is anti-Riemannian with a negative definite metric.
We have the same conclusion if $n =2, n_* \geq 3$
and for $n=n_*=2$ if $M,M_*$ are spheres of constant curvature.
If $n=1$ then $a(t)$ would have to be periodic.
According to Corollary~\ref{f(t,x)} this leads to the same kind
of result. In any case the function $f$ will have a zero.

\medskip
Finally, if $n=n_* = 2$ then according to the proof of Theorem~\ref{pseudoriem}
we have to take the case of two compact surfaces with
nonconstant Gaussian curvatures $K(t)= ca(t),K_*(s)=cb(s)$ into account with a
constant $c > 0$.
The functions $a,b$ would have to satisfy
$$\widetilde{\nabla}^2a = \frac{\widetilde{\Delta}
    a}{2}\widetilde{g} \ \mbox{ and } \ 
{\nabla}_*^2b = \frac{\Delta_*b}{2}g_*.$$
It is well known that the round 2-sphere is the only compact 
surface of that kind, and that the metrics would have to satisfy the equations
$\widetilde{g} = \pm(dt^2 + a'^2(t) dx^2)$ and $g_* = \pm(ds^2 + b'^2(s)dy^2)$
with two exceptional points with a removable singularity
in these coordinates.
The exceptional points correspond precisely to the minima and maxima of the
Gaussian curvatures $K,K_*$.
Assume that one of the metrics is Riemannian and that 
there is such a point $p$ with $K' = a =0$.
Since the singularity in geodesic polar coordinates around $p$ is removable, 
we have necessarily $|a''| = A > 0$ at $p$. This is infinitesimally the
same as for the round sphere or the Euclidean plane
or the hyperbolic plane in polar coordinates.
Consequently, we have $a'' = A$ at the minimum and $a'' = -A$
at the maximum.                            
On the other hand we have $K = -a'''/a'$ and the ODEs                           
$$\textstyle ca^2 + 2a'' = d \ \mbox{ and } \ a'^2 = da - \frac{c}{3}{a^3} + e$$
with constants $d$ and $e$.                                   
The first OED implies $ca^2 = d-2A$ at the minimum and $ca^2 = d+2A$ at
the maximum.                              
Inserting this into the second ODE we obtain        
$$\textstyle 0 = \sqrt{\frac{d-2A}{c}}\Big(d-\frac{d-2A}{3}  \Big)+e$$
%$$0 = 3d \sqrt{(d-2)/c} - (\sqrt{(d-2)/c})^3 + 3e  = \frac{(3c-1)d+2}{c}\sqrt{(%d-2)/c} + 3e$$
at the minimum and
$$\textstyle 0 = \sqrt{\frac{d+2A}{c}}\Big(d-\frac{d+2A}{3}  \Big)+e$$
%$$0 = 3d \sqrt{(d+2)/c} - (\sqrt{(d+2)/c})^3 + 3e  = \frac{(3c-1)d-2}{c}\sqrt{(%d+2)/c} + 3e$$
at the maximum.
Combining these conditions leads to
$$ \sqrt{d-2A}\big( d+A \big)
= \sqrt{d+2A}\big(d-A \big).$$
%The difference of these two equations leads to
%$$((3c-1)d+2)^2(d-2) = ((3c-1)d-2)^2(d+2).$$  
However, there is no real solution $d$ of this equation unless $A = 0$
which would imply that $a$ is constant.
Therefore, no Riemannian manifold $(S^2,\widetilde{g})$ of this kind exists,
and no $(S^2,g_*)$ either.
This is not in contradiction with the fact that abstractly there are 
nontrivial periodic
solutions of the ODE $y^2 + 2y'' = d$ for positive $d$.
However, these necessarily produce surfaces with proper singularities, 
see Figure 1, compare the correction in \cite{Ta81}. \hfill $\Box$

%\begin{figure}[hbt]
%\centering
%\includegraphics[width=80mm]{extremal-color.jpg}
%\caption{extremal surface with a singularity}
%\end{figure}

\begin{figure}[t]
%\begin{SCfigure}
%\centering
%\label{fig:extremal}
%\includegraphics[scale=0.8]{extremal-surface-no-axis-color.jpg}
\includegraphics{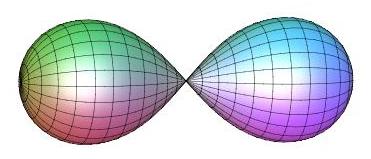}
\caption{Extremal surface as a surface
of revolution in $\R^3$ 
%\newline
satisfying 
$2a''+a^2 = 2$ and 
$a'^2=-a^3/3+2a-4/3.$}
%\end{SCfigure}
\end{figure}

\medskip
{\sc Remark:} The last statement on the product
of two surfaces with nonconstant curvature can be found also in
\cite[Thm.18.4]{De00}, a statement attributed to E.Calabi. 
In the Riemannian case no compact example of $M \times M_*$
is possible at all, except for the product of two curves. 
\begin{cor}{\rm (Ruiz \cite[Thm.1]{Ru09})} \label{Ruiz}

Assume that $n,n_* \geq 2$ and $(M_*^{n_*},g_*)$ is complete and flat, 
and let $(M^n,g)$ be
a compact Riemannian manifold. Then 
$(M,g)\times (M_*,g_*)$
is not globally conformal to any Einstein metric
with a positive Einstein constant.
%unless the product metric itself is already Einstein and $g$ is Ricci flat.
\end{cor}

{\it Proof.}
This follows from Theorem~\ref{pseudoriem} above
if we assume that $f= a+b$ is a globally defined and
strictly positive function.
Since $f$ cannot be constant we assume that the product 
is conformally Einstein
with a non-constant conformal factor $f$.
If $a$ is constant then $g$ is Einstein. By Equation~\ref{sternstern}
$\nabla_*^2b$ is a scalar multiple of $g_*$. Since $g_*$ is flat
we have $\nabla_*^2b=cg_*$ with a constant $c$.
Then Equation~\ref{sternstern} implies the equation $kf = k(a+b) = (N-1)c$.
So either $g$ is Ricci flat ($k =0$) and $c=0$ or $b$ is constant.
In the former case $\nabla_*^2b = 0$ implies that $b$ is either
constant or a linear function, necessarily with a zero.
This leads to a contradiction on a complete manifold. 

\smallskip
If $b$ is constant then by ${\rm Ric}_* = \nabla_*^2b = 0$
Equation~\ref{sternstern} tells us that we have\linebreak
$f^2 {\rm Ric} + (N-1)f \nabla^2a = 0$ and, moreover,
$f\Delta f - N||\grad f||^2 = f\Delta a - N||\grad a||^2$ 
is the Einstein constant $\overline{\lambda}$.
By the compactness of $M$ we have $\int_M\Delta a = 0$,
therefore $- N\int_Mf^{-1}||\grad f||^2 = \overline{\lambda}\int_M f^{-1}$.
Since $f > 0$ we have $\overline{\lambda} < 0$.
%by taking the trace we get $f^2S +(N-1)f\Delta a = f^2S + (N-1)f \Delta
%f = 0$. Together this implies 
%$f^2S-N(N-1)||\grad f||^2$ is constant, hence $f^2k-||\grad f||^2$
%is constant.

\smallskip
If both $a$ and $b$ are not constant 
then $g$ is an Einstein metric and we have $k = -k_* = 0$.
However, the equation $\nabla^2a = cg$ with a constant $c$ has no solution
on a compact manifold except if $a$ is constant and $c=0$.
The case of a nonvanishing parallel gradient of $a$ is impossible.
But then $b$ would be either constant or a linear function.
In the former case $f$ is constant, and 
in the latter case $f = a+b$  has a zero, a contradiction
in any case. \hfill $\Box$
\begin{cor}{\rm (Cleyton \cite[Thm.2]{Cl08})} \label{Cleyton}

Assume $n,n_*\geq 2$, and let $(M^n,g)$ and $(M_*^{n_*},g_*)$
be complete Riemannian manifolds. If $M\times M_*$ is globally
conformal to an Einstein space with metric $f^{-2}(g+g_*)$ 
that is not a warped product, then
$M$ and $M_*$ are Euclidean spaces, and the conformal factor $f$
is $f(y,x) = \frac{1}{2}(||x||^2+||y||^2) + d^2$ with a constant $d
\not = 0$. 
\end{cor}
{\it Proof.} The proof is somehow contained in the proof of
the preceding Corollary. Since the conformal product is assumed
not to be a warped product, we have only the third case with a
function $f=a+b$ where both $a$ and $b$ are non-constant.
By the equation $\widetilde{k} = -k_*$ we have two cases: $k_* = 0$ and
$k_* \not = 0$. In the latter case $k_* > 0$
one of the equations reads as $b'' = -k_*b + c$. This implies that
$M_*$ is a sphere of sectional curvature $k_*$ and that 
$b(s)= \sin s + c$
along the trajectories of its gradient. This statement is
known as the theorem of Obata \cite{Ob62} and Tashiro \cite{Ta65}.
Consequently we have the equation $a'' = k_*a + c$, so $a$ is
a hyperbolic function minus $c$ along the trajectories of its gradient.
In any case $a(t) + b(s)$ has a zero on the complete space $M\times
M_*$.
There remains the former case $\widetilde{k}=k_*=0$. In this case
$a$ and $b$ satisfy the equations $\widetilde{\nabla}^2a =
c\widetilde{g}, \nabla_*^2b=cg_*$ with a constant $c$.
Then $a$ and $b$ are quadratic polynomials along the trajectories
with the leading coefficient $c/2$.
If $c=0$ then both are linear and, therefore, have zeros on
$M\times M_*$.
Furthermore it is well known that the Euclidean space is the only
complete Riemannian manifolds admitting a non-constant and
everywhere positive function $a$
with $\nabla_*^2b = cg_*$ and a constant $c \not = 0$. In polar
coordinates $g_* = ds^2 + b'^2(s)g_1$ around the origin 
we have $c^2=1$ and $b(x) =
\frac{1}{2}||x||^2$
plus a positive constant, similarly for $a$.
The assertion follows. \hfill $\Box$
\begin{cor} {\rm (the ``improper case''  
in the terminology of Brinkmann \cite{Br25})}\label{improper}

Under the same assumptions as for Theorem~\ref{pseudoriem}
the following holds:
If in addition $\grad f$ is a null vector field (i.e., if $||\grad  f||^2 =
0$) on an open subset
then either one of the following three cases occurs:
\begin{enumerate}
\item $(M \times M_*,\overline{g})$ is a warped product
with a Ricci flat Brinkmann space as the base,
\item
$(M,\widetilde{g})$ and $(M_*,g_*)$
are Ricci flat Brinkmann spaces carrying a parallel null vector field
each.
On the other hand the product metric $g = \widetilde{g}+g_*$
itself is Ricci flat,
so we have a case of Corollary~\ref{Brinkmann} also. 
\item Both admit a parallel non-null vector field, one
spacelike, the other timelike.
The latter case implies that in $M$ and $M_*$ both
are pseudo-Riemannian products of $\R$ with a Ricci flat manifold. Hence in
the product a $2$-dimensional flat factor splits off.
\end{enumerate}
Consequently, if $(M\times M_*,\widetilde{g}+g_*)$
is not Ricci flat then we have Case (1), i.e.,
$\overline{g}$ is a warped product metric
with a Ricci flat Brinkmann space as base but the fiber is not
Ricci flat.
By \cite[Cor.9.107]{Be87} this case cannot occur:
We would necessarily have $f = u$ with the lightlike 
${\rm grad} f = \partial_v$ on the
Brinkmann base, furthermore $\overline{\lambda} =0$ and an
Einstein $(M_*,g_*)$ with $\lambda_* = 0$.  
Hence the product metric is always Ricci flat. 
\end{cor}

{\it Proof:}
By Theorem~\ref{pseudoriem} 
either case (1) occurs or 
$M,M_*$ are Einstein spaces admitting solutions of
$\widetilde{\nabla}^2a = \frac{\widetilde{\Delta}a}{n}\widetilde{g}$
and $\nabla_*^2b = \frac{\Delta_* b}{n_*}g_*$ with $f = a+b$
and $||\widetilde{\grad} a||^2 = -||\grad_* b||^2$.
Since one depends only on $M$, the other one only on $M_*$,
both must be constant.
If both are zero then Case (2) occurs by \cite[Thm.3.12]{KR09}.
Otherwise both $\widetilde{\grad} a$ and $\grad_* b$
are parallel vectors, one of them spacelike, the other one timelike.
\hfill $\Box$

\begin{exa}\label{ppexamples} \rm Case (2):
A $pp$-wave with metric $$ds^2 = -2H(u,x,y)du^2 -2dudv + dx^2 + dy^2$$ 
is Ricci flat if and only if the spatial Laplacian
$\Delta H = H_{xx}+H_{yy}$ vanishes \cite{EM86},\cite{KR04}.
The vector field $\grad u = \partial_v$ is a parallel null vector
field, thus the
equation $\nabla^2u = 0$ is satisfied.
If we take the cartesian product of two such $pp$-waves
with parameters $u_i,v_i,x_i,y_i$, $i=1,2$,
then the function $f = u_1+u_2$ satisfies all the
assumptions in Theorem~\ref{pseudoriem}.
This is an example of Case (2) in Corollary~\ref{improper}.

\medskip
Case (1): If $g_*$ is a Ricci flat $pp$-wave as in Case (2)
then with $M = \R$ and $f(t,x) = t+1$ the
warped product $\overline{g} = f^{-2}(dt^2 + g_*)$
is an Einstein space with Einstein constant $\overline{\lambda} = -1$.
This corresponds to one of the cases discussed in 
Proposition~\ref{f(t)}.

\medskip
Case (3): The products $\R \times M$ with metric $dt^2+g$ and $\R \times M_*$
with metric $-d\tau^2+g_*$ where both $g,g_*$ are Ricci flat
provide an example
by taking $\R \times M \times \R \times M_*$ with the product
metric $ds^2 = dt^2 + g -d\tau^2 + g_*$.
The function is $f(t,\tau) = t + \tau$.
Obviously a factor $\R^2$ with the Lorentzian metric $dt^2 -d\tau^2$
splits off.
\end{exa}
\begin{cor}\label{f(t,x)}
 {\rm (Conformally Einstein products of type $\R \times M_*$)}

If $f(t,x)$ is a non-constant function of $t\in \R$ and $x \in M_*$
with an $m$-dimensional pseudo-Riemannian manifold $(M_*,g_*)$, $m\geq 2$, 
and if the metric $\overline{g} = f^{-2}(\epsilon dt^2 + g_*)$ is Einstein 
then one of the following cases occurs:

\begin{enumerate}
\item $f$ depends only on $t$ (the case discussed in Proposition~\ref{f(t)}) 
\item $f$ depends only on $x$ (the case discussed in Proposition~\ref{conf-product})
\item $f(t,x) = a(t) + b(x)$ with non-constant $a$ and non-constant $b$,
and $(M_*,g_*)$ is an $m$-dimensional Einstein space
with constant normalized scalar curvature $k_*$ (even in dimension $2$), 
and $a$ satisfies the equation $a'' = \epsilon k_*a + c$ for a constant $c$,
and simultaneously 
$b$ satisfies the equation $\nabla_*^2b = \frac{\Delta_* b}{m}g_*$.
Such Einstein spaces can be (locally and globally) classified
\cite{Ku88}, \cite{KR09}.
\end{enumerate}
\end{cor}

\begin{proof}
This is nothing 
%but 
%%%%%%%%%%%%%%%%%changes
else than a special case of Theorem~\ref{pseudoriem}.
Independently, the calculation is a bit simpler as follows:
Let $g = \epsilon dt^2+ g_*$. Then by Equation~\ref{eq:ric}
$$\overline{\Ric} - \Ric = (m-1) f^{-1} \cdot 
\nabla^2 f + \Big[f \cdot \Delta f - m \cdot 
\|\grad f\|^2 \Big]\cdot \overline{g} $$ 
the Einstein condition for $\overline{g}$ 
implies that $\nabla^2 f$ has the form
\[ \nabla^2 f = \left( \begin{array}{cc}
\frac{\partial^2f}{\partial t^2}&0\\0&*\end{array} \right)\] 
with some $(0,2)$-tensor on $M_*$ denoted by $*$.
In coordinates $x_1, \ldots,x_n$ on $M_*$ this implies
$\frac{\partial^2f}{dtdx_i} = 0$ for any $i$.
Therefore $f$ splits as
$$f(t,x) = a(t) + b(x)$$ 
with functions $a$ of the real variable $t$ and $b$ of $x \in M_*$.
This implies
\[ \nabla^2 f = \epsilon a'' (\epsilon dt^2) + \nabla^2_*b = \left( \begin{array}{cc}
\epsilon a''&0\\0&\nabla_*^2 b\end{array} \right)\] 
and 
$$||\grad f||^2 = \epsilon (a')^2 + ||\grad_* b||^2, \quad  \Delta f =
\epsilon a'' + \Delta_* b.$$
From 
\[\left( \begin{array}{cc}
\epsilon a''&0\\0&\nabla_*^2 b\end{array} \right)=
\epsilon a''f^2 \overline{g} + \left( \begin{array}{cc}
0&0\\0&\nabla_*^2 b-\epsilon a''g_*\end{array} \right)
\]
we see that $\overline{g}$ is Einstein if and only if
$$(a+b)\Ric_* + (m-1)\Big(\nabla_*^2 b-\epsilon a''g_*\Big) = 0.$$
Splitting the $t$-part from the $x$-part leads to
\begin{equation}
\label{eq:ric*}
a\Ric_*-(m-1)\epsilon a''g_* = -b\Ric_* -(m-1)\nabla_*^2b
\end{equation}
with the trace
$$aS_* - m(m-1)\epsilon a'' = -bS_* - (m-1)\Delta_* b.$$
Differentiating by $t$ leads to
$$a'S_* - m(m-1)\epsilon a''' = 0,$$
hence $S_* = m(m-1)k_*$ must be constant (unless $a$ is constant)
since it depends only on $x$.
It also follows that $a$ is a solution of the standard ODE
$a''' = \epsilon k_*a'$, in particular $a'' = \epsilon k_*a + c$ with a constant $c$
and, by the trace of Equation~\ref{eq:ric*}, with $c = k_*b + \Delta_*b/m$.
In particular $\Delta_*b$ depends only on $b$.
By inserting this into Equation~\ref{eq:ric*}
we obtain
\begin{equation}
\label{eq:ric*new}
a\Big(\Ric_*-\frac{S_*}{m}g_*\Big)=(m-1)cg_* -b\Ric_* -(m-1)\nabla_*^2b
\end{equation}
Since only the left hand side depends on $t$ but the right hand
side does not, it follows that either
$a$ is constant (then the constant can be incorporated into $b$,
and we can apply Corollary~\ref{conf-product}) or
$g_*$ is an Einstein metric and the right hand side is zero.

\medskip
\underline{Case 1:} $c=0$ and $b=0$. This is discussed in 
Proposition~\ref{f(t)}.
We recognize the key equation $a'' = k_*a$ from the proof of
Proposition~\ref{f(t)}.
 
\medskip
\underline{Case 2:} $c=0$ and $b\not =0$.
If $a$ is constant we have the case of Proposition~\ref{conf-product}.
If $a$ is not constant then $g_*$ is Einstein and $b\Ric_* + (m-1)\nabla_*^2b
= 0$. It follows that $\nabla_*^2b$ is also a scalar multiple of
$g_*$,
hence $\nabla_*^2b = \frac{\Delta_* b}{m}g_*$.

\medskip
\underline{Case 3:} $c\not =0$. 
If in addition $a$ is not constant then $\Ric_*$ is a multiple of $g_*$
and, consequently, $\nabla_*^2b$ is also a scalar multiple of $g_*$, hence
$\nabla_*^2b = \frac{\Delta_* b}{m}g_*$ as in Case 2.
On an Einstein space this equation can be completely and
explicitly solved, see \cite{Ku88}, \cite{KR09}.
\end{proof}
%\hfill $\Box$

%The trace of the remaining equation
%$$\overline{\Ric} = (a+b)\Delta(a+b) - n (a')^2 - n
%||\grad_*b||^2 + (n-1)(a+b)a''$$ leads to
%$$\frac{\overline{S}}{n} = (a+b)(a''+\Delta_*b) - n(a')^2 - n||\grad_*
%b||^2 + (n-1)aa''+ (n-1)ba''.$$
%Taking the derivative by $t$ leads to

\begin{cor} {\rm (Moroianu \& Ornea \cite[Thm.2.1]{MO})}\label{Moroianu}

If $M_*$ is 
%%%%%%%%%%%changes
a 
compact Riemannian manifold
and if $\overline{g} = f^{-2}(dt^2+g_*)$ is
Einstein
with $\overline{S} > 0$
and with a non-constant 
function $f(t,x)$ that is globally defined 
and never zero on $\R \times M_*$ 
then one of the following two cases occurs:
\begin{enumerate}
\item
$(M_*,g_*)$ is a round sphere,
\item
$(M_*,g_*)$ is an Einstein space with $S_* > 0$, and the function $f(t,x)$
is the $\cosh$-function on the real $t$-axis, up to constants. 
In particular $f$ does not depend on $x \in M_*$.
\end{enumerate}
\end{cor}
{\sc Remark:} The compactness of $M_*$ will be used only for Case 1.
In fact that assumption is not essential for Case 2, and our results above
show that the question of the authors of \cite{MO} at the end 
of Section 1 can be answered. In general the statement
of Case 1 would have to be modified, compare Corollary~\ref{new}
below. See also Example~\ref{non-standard-mercator}.
It would be even sufficient to assume that $f$ is defined
on an open set containing at least one
slice $\{t_0\} \times M_*$.

\medskip
{\it Proof.} If $f(t,x)$ depends on $t$ and on $x$ then
$(M_*,g_*)$ is a round sphere by Proposition~\ref{f(t,x)}
in combination
with the well known theorem that
the only compact pseudo-Riemannian Einstein space
admitting a non-constant solution of the equation 
$\nabla_*^2b = \frac{\Delta_* b}{n}g_*$ is the round sphere
\cite{Ob62}, \cite{Ta65}, \cite{Ka83}, \cite{Ku88}, \cite[Thm.2.8]{KR09}.

\medskip
If $f$ depends only on $t$ then by Proposition~\ref{f(t)} $g_*$ is Einstein.
Furthermore the case $\overline{S} > 0$ 
is only possible for a function of $\cosh$-type
(up to additive or multiplicative constants). This 
is a consequence of the ODE $ff''-(f')^2 = \overline{k} > 0$. 
Moreover from $f'' = k_*f$ we get $k_* > 0$. 

\medskip
If $f$ depends only on $x$ then by Proposition~\ref{conf-product}
$k_*$ is negative. Then Equation~\ref{eq:trace} implies that
$\Delta_*f$ is positive
at a positive maximum of $f$, and it is negative at a negative
minimum of $f$. Here we use that $f$ never vanishes on $M_*$.
This is a contradiction on a compact manifold $M_*$
because one of these cases must occur.
However, compare Example~\ref{example:sphere} with the opposite signs
$k_*>0$ and $\overline{k}<0$.
\hfill $\Box$

\bigskip
We can prove a similar version under a slightly different assumption:
\begin{cor} \label{final}
If $M_*$ is compact Riemannian of constant scalar curvature $S_*$
and if $\overline{g} = f^{-2}(dt^2+g_*)$ is
Einstein with a non-constant function $f(t,x)$ that is globally
defined and never zero on $\R \times M_*$, 
then one of the following two cases occurs:
\begin{enumerate}
\item
$(M_*,g_*)$ is a round sphere, and $\overline{g}$ is
of constant sectional curvature,
\item
$(M_*,g_*)$ is an Einstein space with $S_* > 0$, 
and the function $f(t,x)$
is the $\cosh$-function or the exponential function
on the real $t$-axis, up to constants.  
We have $\overline{S} > 0$ in the first case and $\overline{S} = 0$
in the second case.
%In particular $f$ does not depend on $x \in M_*$.  
%In particular $f$ does not depend on $x \in M_*$.
\end{enumerate}
\end{cor}
{\it Proof.}
The first case is the same as in Corollary~\ref{Moroianu}: 
If $f$ depends on $t$ and on $x$
then $(M_*,g_*)$ is a round sphere.
Therefore the product metric on $\R \times M_*$
is locally conformally flat. It can even be realized as a
hypersurface in Euclidean space, namely, as a tube around a straight line.
Therefore $\overline{g}$ is of constant sectional curvature, compare
Example~\ref{flatexample}.

\medskip
In the second case ($f$ depends only on $t$)
we have to find all never vanishing global 
solutions of the ODE $ff'' -(f')^2 = \overline{k}$.
These are the cases of the exponential function (up to constants)
if $\overline{k} = 0$ and hence $S_* > 0$ or
the $\cosh$-function (up to constants) if $\overline{k} > 0$ 
and hence $S_* > 0$. The case $\overline{k} < 0$ cannot occur.

\medskip
In the third case ($f$ depends only on $x$) 
we see from Equation~\ref{eq:trace} that $\varphi$ is
an eigenfunction of the Laplace operator on $(M_*,g_*)$
for the eigenvalue $-S_*/(n-1)$. This implies $S_* > 0$ and
$$\int_{M_*}\varphi = -\frac{n-1}{S_*}\int_{M_*}\Delta_*\varphi = 0.$$
This is a contradiction because by assumption $\varphi$ has no zeros.
\hfill $\Box$

\begin{exa}\rm \label{non-standard-mercator}
A concrete example for Corollary~\ref{Moroianu} and
Corollary~\ref{final} is the following:
Let $(M_*,g_*)$ be a compact Einstein space with a positive
Einstein constant and normalized scalar curvature $k_* = 1$, 
e.g., a complex projective space.
Then the metric $\overline{g} = (\cosh t)^{-2}(dt^2+g_*)$ is Einstein with 
normalized scalar curvature $\overline{k} = 1$. This is not of
constant curvature if $g_*$ is not of constant curvature.
The resulting metric $\overline{g}$ coincides with the
metric of the classical Mercator projection
if $g_*$ is the metric of the unit sphere, compare Example~\ref{Mercator}.
The other cases can be called a {\it non-standard Mercator projection}
on the product $\R \times M_*$. This is particularly interesting if
$M_*$ is a sphere and $g_*$ is a non-standard Einstein metric
(i.e., not of constant sectional curvature), see \cite{Bm98},
\cite{BGK}. 
In these cases
a compactification by two points (north and south pole) is
homeomorphic with the sphere $S^n$, and it is Einstein except at
these two points which are metrical singularities.

\end{exa}

Without the assumption of compactness we have the following:
\begin{cor}\label{new} 
If $\overline{g} = f^{-2}(dt^2+g_*)$ is an $n$-dimensional 
Riemannian, complete and Einstein metric 
with a non-constant function $f(t,x)$ that is globally
defined and never zero on $I \times M_*$ with an open interval $I \subseteq
\R$, 
then one of the following two cases occurs:
\begin{enumerate}
%\item
%$(M_*,g_*)$ is a round sphere,
\item 
$f$ depends only on $x$ (the case of Corollary~\ref{conf-product}). 
\item
$(M_*,g_*)$ is a complete Einstein space with $S_* \leq 0$, 
and the $f$ depends only on $t$ (possibly
after a 
%chance
%%%%%%%%%%%%changes
change of variables). Moreover, $f$
is either $f(t) = \alpha \cos t + \beta \sin t$ on a bounded interval
$I$, or $f(t) = \alpha t + \beta$ on an unbounded 
interval $I$, each time with constants $\alpha, \beta$.  
We have $\overline{S}, S_* < 0$ in the first case and $\overline{S}
<0, S_* = 0$ in the second case.
\end{enumerate}
Particular cases are the hyperbolic metric $\overline{g}$  in the form
$\overline{g} = \sin^{-2}t(dt^2 + g_{-1})$ (compare Example~\ref{Mercator})
or $\overline{g} = t^{-2}(dt^2 + g_0)$ (the Poincar\'{e} halfspace).
\end{cor}
{\it Proof.}
If $f(t,x) = a(t) + b(x)$ depends on both arguments 
with nonconstant functions $a$ and $b$ then 
$(M_*,g_*)$ is complete Einstein with $\nabla_*^2b = \frac{\Delta_* b}{n-1}g_*$
and simultaneously $a'' = k_* a + c$. If in addition $k_* > 0$ then
the solution $a(t)$ is a hyperbolic function. 
Simultaneously $b$ is bounded. Therefore
$\overline{g}$ cannot be complete in the $t$-direction by the growth of $a$
at infinity.
The case of the Mercator projection shows that $\overline{g}$ can be
extendible to a complete (even compact) manifold.
If $k_* < 0$ then $a$ is bounded and
$b$ is a hyperbolic function on $M_*$ which implies
that $\overline{g}$ cannot be complete.
If $k_* = 0$ then $a$ is linear and $b$ is either 
linear or quadratic with a zero. In the latter case 
$\overline{g}$ is not complete. 
It remains the case that $a$ is linear in $t$ and $b$ is linear
in a variable $u$ such that $g_*$ is a Ricci flat 
warped product $g_* = ds^2 + b^2(s)g_{**}$ where $g_{**}$ is again
Ricci flat.
But then $f = a + b$ is linear in the $(t,s)$-plane,
and we can regard $g = dt^2 + g_*$ as a metric of type
$g = du^2 + dv^2 + g_{**}(u)$ such that
$\overline{g} = u^{-2}g$. This is part of case (2)
with an unbounded interval $I$.

\medskip
If $f$ depends only on $t$ then we are in the position of
Proposition~\ref{f(t)}. From the list in the proof there
we see that $\overline{g}$ cannot be complete unless
$f$ is linear of $f(t)= \sin t$ or $f(t) = \cos t$ or a linear combination.
Since $f$ has a zero in any case, the interval $I$ 
is unbounded in the linear case but must be bounded in the other case.
This is part of Case (2).
%\medskip
%If $f$ depends only on $x\in M_*$
%then we have three cases:
%Case 1: $S_* \equiv 0$, $||\grad f||$ is constant, and $f$ is
%an affine function with a parallel gradient, and $g_*$ is
%Ricci flat. But then $\overline{g}$ is not complete, contradictin
%our assumption. 
%Case 2: $S_*$  
%Case 3: $S_*$ 
\hfill $\Box$

\begin{exa}\rm (a compact example)

Let $(M,g)= (S^n,g)$ be the unit sphere with normalized scalar curvature
$k=1$.
Take a second copy $(M_*,g_*) = (S^n,-g)$ with a negative definite metric
and with normalized scalar curvature $k_* =-1$.
On $M$ the function $a(t) = \cos t$ satifies the equation
$\nabla^2 a = -kag = -ag$, on $M_*$ the function 
$b(s) = \cos s $
satisfies the equation $\nabla_*b = -k_*bg_* = bg_*$.
By the results of Theorem~\ref{pseudoriem}
the metric $\overline{g} = (a+b)^{-2}(g + g_*)$ defines an Einstein
space whenever $a+b \not = 0$. The Einstein constant is
$$\overline{\lambda} = (n-1)(a+b)^2 + 
(2n-2)(a+b)(-a) + (a+b)n(b-a)-(2n-1)(a'^2-b'^2)
$$
$$= (a+b)(-(2n-1)(b-a)) -(2n-1)(a'^2-b'^2) = (2n-1)(-(a^2+a'^2) +
(b^2+b'^2)) = 0.$$
Unfortunately the function $f(t,s)=a(t)+b(s)$ has zeros. This means
that $M\times M_*$ ist not globally conformally Einstein. It would not
help to introduce an additive constant $c$ here because then we would have
$a(t)=\cos t +c$ and $b(s)=\cos s -c$.
\end{exa}

\begin{exa} \rm
Here is an example of a complete Riemannian manifold
$$M = \R \times \R \times \widetilde{M} = \R \times M_*$$ admitting a global
solution $f = a(t) + b(x)$ where both $a,b$ are non-constant and never zero.
Unfortunately $f$ has zeros.
Let $(\widetilde{M},\widetilde{g})$ be a complete Ricci flat
manifold of dimension $n-1$, and let
$$a(t) = \cos t - 2, \quad b(s) = e^s + 2.$$
Then on $M$ the function 
$$f(t,s,\widetilde{x}) = a(t) + b(s)$$
satisfies all conditions above:
The metric $g_* = ds^2 + e^{2s}\widetilde{g}$ on $M_* = \R \times
\widetilde{M}$ is a complete Einstein space with $k_* = -1$
\cite{Ku88} and $a,b$ satisfy the equations
$$a'' = k_*a - 2, \quad b'' = -k_*b - 2, \quad \nabla_*^2b = (-k_*b -2)g_*,
\quad c=-2.$$
By the results above the metric $\overline{g} = f^{-2}\big(dt^2 + ds^2 +
e^{2s}\widetilde{g}\big)$ on $M = \R \times M_*$ 
is Einstein with $\overline{k} = -1$ whenever $f(t,s) = \cos t + e^s \not = 0$.
If $\widetilde{g}$ is not flat then $g_*$ is not of constant curvature
and, consequently, $\overline{g}$ is not of constant curvature.

\end{exa}

\begin{exa} \rm
A similar example starts with an $(n-1)$-dimensional Einstein space
$(\widetilde{M},\widetilde{g})$ with normalized scalar curvature
$-1$. Then $M_* = \R \times \widetilde{M}$ with the
warped product metric $g_* = ds^2 + \cosh^2(s)\widetilde{g}$ is also Einstein
with $k_* = -1$ \cite{Ku88}. Let
$$a(t) = \cos t, \quad b(s) = \cosh s.$$
Then on $M = \R \times M_*$ the function 
$$f(t,s,\widetilde{x}) = a(t) + b(s)$$
satisfies all conditions above since $a,b$ satisfy the equations
$$a'' = k_*a, \quad b'' = -k_*b, \quad \nabla_*^2b = (-k_*b)g_*,
\quad c=0.$$
By the results above the metric $\overline{g} = f^{-2}\big(dt^2 + ds^2 +
\cosh^2(s)\widetilde{g}\big)$ on $M$ 
is Ricci flat whenever $f(t,s) = \cos t + \cosh s \not = 0$.
This is the case at least for all $s\not = 0$,
If $\widetilde{g}$ is not hyperbolic then $g_*$ is not of constant curvature
and, consequently, $\overline{g}$ is not of constant curvature.

\end{exa}
\begin{exa} \label{flatexample} \rm
Here is an example with a compact $M_*$ and a function depending
on two parameters:
Let $(M_*,g_*)$ denote the unit $n$-sphere in polar coordinates\linebreak
$g_* = dr^2 + \sin^2(r)\widetilde{g}$ with the unit $(n-1)$-sphere
$(\widetilde{M},\widetilde{g})$.
This admits the function $b(r,\widetilde{x}) = \cos r$ satisfying $\nabla_*^2b =
-bg_*$.
Then on $M = \R \times M_*$ the function 
$$f(t,r,\widetilde{x}) = a(t) + b(r,\widetilde{x}) = \cosh t + \cos r$$
satisfies all conditions above:
The metric $g_*$ is a complete Einstein space with $k_* = 1$,
and $a,b$ satisfy the equations
$$a'' = k_*a, \quad b'' = -k_*b, \quad \nabla_*^2b = (-k_*b)g_*,
\quad c=0.$$
By the results above the metric $\overline{g} = f^{-2}\big(dt^2 + dr^2 +
\sin^2(r)\widetilde{g}\big)$ on $M = \R \times M_*$ is
%Einstein with Einstein constant $\overline{\lambda} = 0$
Ricci flat whenever $f(t,s) \not = 0$. This is the case for
$t \not = 0$ or $r \not = \pi$, i.e. on $(\R \times M_*) \setminus
\{(0,p)\}$ for a point
$p$. $\overline{g}$ is even flat according to Corollary~\ref{final}.
%It seems that $\overline{g}$ is not of constant curvature.
Compare the hyperbolic metric as a warped product
$dt^2 + \cosh^2(t)\big(dr^2 + \sinh^2(r)\widetilde{g}\big).$
\end{exa}
%\begin{cor}\label{final}
%Let $(M_*,g_*)$ be compact Riemannian of dimension $m \geq 2$, 
%and let $f(t,x) = a(t) + b(x)$ be a function
%with nonconstant $a$ and $b$ such that $\overline{g} = f^{-2}(dt^2+g_*)$
%is Einstein. Then $(M_*,g_*)$ is isometric with a round sphere,
%and consequently the product $\R\times M_*$ is locally conformally flat.
%If $\overline{g}$ is Ricci flat then it is flat.
%\end{cor}

%{\it Proof.} The proof follows from Corollary~\ref{f(t,x)}.
%The round sphere is the only compact Einstein space that admits
%a nonconstant solution $b$ of $\nabla^2b = \frac{\Delta b}{m}$.

\section{Conformally Einstein products of type
    $\R \times M_*$}

In this section we come back to some particular cases that were not
solved so far. In Proposition \ref{f(t)} we discussed the case of
a function $f$ depending on one real parameter. Here we deal with
the opposite case:
A function depending on the $(n-1)$-dimensional factor $M_*$ 
of the product $\R \times M_*$.
This leads to equations involving the Hessian of the function and
the Ricci tensor. If one of them (and, consequently, the other as well) is
a multiple of the metric, then we are in the position of 
%%%%%%%%%%%%%%%changes
Theorem \ref{Brinkmann}. The question is whether there are more
solutions. In fact they are: We construct iterated warped products
of this kind by solving a nonlinear ODE of third order.

%\pagebreak
\begin{pro} \label{conf-product}
{\rm (The type
    $\R \times M_*$ with an $(n-1)$-dimensional base)}

If $g$ has the form $g = \pm dt^2 + g_*$ on $M = \R \times M_*$ and if 
$\varphi$ is a never vanishing function defined on the $n$-dimensional manifold
$(M_*,g_*)$ then $\overline{g}= \varphi^{-2}g$ is an
Einstein metric if and only if the two equations
\begin{eqnarray}
\label{eq:conf-Ein}
\varphi \cdot \Ric_* + (n-1) \cdot \nabla^2_*\varphi&=&0 \\
\|\grad \varphi\|^2 + \varphi^2k_* + {\overline{k}}&=&0
\end{eqnarray}
are satisfied with a constant $\overline{k}$. 
Here $k,k_*, \overline{k}$ denote the normalized
scalar curvatures of the metrics $g,g_*,\overline{g}$.
The trace of the first equation is the following:
\begin{equation}
\label{eq:trace}
\varphi \cdot S_* + (n-1) \cdot \Delta_*\varphi
=0  
\end{equation}
\end{pro}
{\it Proof.} The first equation follows directly 
from Equation~\ref{eq:ric} for the $(n+1)$-dimensional manifold $M$
by the block matrix structure of $\Ric$ and $\nabla^2\varphi$
\[\Ric = \left( \begin{array}{cc}
0&0\\0&\Ric_* \end{array} \right), \quad \nabla^2 \varphi = \left( \begin{array}{cc}
0&0\\0&\nabla_*^2 \varphi \end{array} \right)\] 
and 
$\overline{\Ric}= n\overline{k}\overline{g}$.
Then the trace of Equation~\ref{eq:ric} leads to the equation
$$\overline{S} = (n+1)\big[\varphi \Delta_*\varphi - n ||\grad_* \varphi||^2\big].$$
By inserting Equation~\ref{eq:trace} and $\overline{S} =
n(n+1)\overline{k}, \ S_* = n(n-1)k_*$ we obtain the equation
$\overline{k} = -\varphi^2k_* - ||\grad_*\varphi||^2.$
\hfill $\Box$

\begin{exa}
\label{example:sphere}
\rm
A standard example for Corollary~\ref{conf-product}
is the unit sphere $(M_*,g_*)$
with a linear height function $\varphi(r) = \cos r$ in polar coordinates.
Here we have the equations 
$\Ric_*= (n-1)g_*, \ \nabla_*^2\varphi = -\varphi g_*$
and $||\grad_*\varphi||^2 = \sin^2 r, 
\ k_* = 1, \ \overline{k} = -1$.
However, the conformally transformed metric $\overline{g}$ is only
defined on the part of $\R \times S^n$ where $\cos r \not = 0$
(as in the case of the Mercator projection).
\end{exa}
\begin{cor} {\rm (Corvino \cite[Prop.2.7]{Co00})}
\label{Corvino}

With the same notations as in Corollary~\ref{conf-product} and with
a function $f$ on $M_*$ the metric $\overline{g}= \pm f^2 dt^2 + g_*$ is Einstein 
if and only if the equation
\begin{equation}
\label{eq:Corvino}
f \cdot \Ric_* - \nabla^2_*f + \Delta_* f \cdot g_* = 0
\end{equation}
holds. Its trace 
\begin{equation}
\label{eq:tracef}
f \cdot S_* + (n-1) \cdot \Delta_*f
=0,
\end{equation}
coincides with Equation~\ref{eq:trace}.

\end{cor}
{\it Proof.}
The metric $\overline{g}$ is a warped product metric
with an $n$-dimensional base $(M_*,g_*)$ and a 1-dimensional fibre.
It can be regarded as a conformal transformation of the product metric
$g = \pm dt^2 + f^{-2}g_*$. In this way Equation~\ref{eq:Corvino} follows
from Equation~\ref{eq:ric}:
For the metric $\widehat{g} = f^{-2}g_*$ we have
$$f \cdot \widehat{\nabla}^2 f = f \cdot \nabla_*^2 f- ||\grad_* f||^2 g_* +
2df \otimes df$$
and
$$\widehat{\Ric} = \Ric_* + \frac{n-2}{f}\nabla^2_* f + 
\big[f \Delta_*f - (n-1) ||\grad_* f||^2\big] \cdot \widehat{g}$$ 
Each quantitiy for the metric $g$ is obtained from the same quantity 
for $\widehat{g}$ as a block matrix with a $0$ entry in the $t$-component:
\[\Ric = \left( \begin{array}{cc}
0&0\\0&\widehat{\Ric} \end{array} \right), \quad \nabla^2 f = \left( \begin{array}{cc}
0&0\\0&\widehat{\nabla}^2 f \end{array} \right)\] 
For the conformally transformed metric $\overline{g} = f^2 g =
f^2(\pm dt^2 + \widehat{g})$ we obtain
$$\overline{\Ric} = \Ric + (n-1)f\nabla^2 f^{-1} + 
\big[f^{-1} \Delta f^{-1} - n ||\grad f^{-1}||^2\big] \cdot \overline{g}$$
By combining this with the equation
$$\nabla^2f^{-1} = 2f^{-3} df \otimes df - f^{-2}\nabla^2f$$
%and 
%$$\Delta f^{-1} =  2 f^{-3}||\grad f||^2 - f^{-2}\Delta f, \quad
%\grad f^{-1} = -2f^{-2} \grad f$$
we obtain that $\overline{\Ric}$ equals some scalar
multiple of $\overline{g}$ plus
\[\left( \begin{array}{cc}
0&0\\0&\Ric_* \end{array} \right) + \left( \begin{array}{cc}
0&0\\0&f^{-1}\cdot \Delta_* f \cdot g_* \end{array} \right) - \left( \begin{array}{cc}
0&0\\0&f^{-1}\cdot\nabla^2_* f \end{array} \right).\]
The assertion follows. \hfill $\Box$
\begin{lem} {\rm (Bourguignon \cite{Bo75}, Corvino \cite[Prop.2.3]{Co00})}

Assume that the equation 
$f \cdot \Ric - \nabla^2f + \Delta f \cdot g = 0$
holds for a never vanishing function $f \colon (M^n,g) \to \R$. Then the metric
$g$ is 
%(locally) 
%%%%%%%%%%%%%%%%%%%%%%changes
of constant scalar curvature $S$.
Moreover the equation $\Delta f = -\frac{S}{n-1} \cdot f$ holds. 
\end{lem}

{\it Proof.}
From the Ricci identity for two commuting vector fields $X,Y$
$$\nabla_X\nabla_Y \grad f - \nabla_Y\nabla_X \grad f = R(X,Y) \grad f$$
we obtain the trace
$${\rm tr}(\nabla\nabla_Y \grad f) - \nabla_Y({\rm tr}\nabla \grad f) = \Ric\cdot
\grad f (Y)$$
or
$${\rm div}(\nabla^2f) - d(\Delta f) = \Ric\cdot \grad f.$$
If we combine this with the divergence of the equation
$f \cdot \Ric - \nabla^2f + \Delta f \cdot g = 0$
then we obtain
\begin{eqnarray*}
%\label{eq:conf-Ein}
0 & = & {\rm div}(f\cdot\Ric) - {\rm div}(\nabla^2 f) + {\rm div}(\Delta f
\cdot g) \\
&=&f\cdot {\rm div} (\Ric)+ \Ric \cdot \grad f - {\rm div}(\nabla^2 f)
+ d(\Delta f)\\
&=&f\cdot {\rm div} (\Ric)\\
&=&\textstyle f \cdot d(\frac{1}{2}S)
\end{eqnarray*}
where the last equality follows from the
second Bianchi identity ({\it ``the Einstein tensor is divergence free''})
\cite[6.15]{Ku02}. The condition $f \not = 0$ implies the assertion
$dS = 0$. A discussion of possible zeros of $f$ is given in \cite{Co00}.
The equation $\Delta f = -\frac{S}{n-1} \cdot f$ follows from the
trace of the other equation, compare Equation~\ref{eq:tracef}.
\hfill $\Box$

%\medskip
%{\sc Remark:} The assertion of this Lemma does not hold if $n=2$.
%If $K= -f'''/f'$ is the Gaussian curvature on a surface
%with warped product metric $dt^2+(f'(t))^2dx^2$
%then we have $\nabla^2f = f'' g$.
%Therefore the equation is just
%$f{\rm Ric}-\nabla^2f+2\Delta f g=(-f f'''/f' + f'')g = 0$ or,
%equivalently, $ff'''-f''f' = 0$.
%This ODE admits non-constant solutions.

\begin{cor}
If $f \colon M_* \to \R$ is a function on a 
%connected 
%%%%%%%%%%%%%%%%%%changes
$n$-dimensional manifold
$(M_*,g_*)$ such that
$\overline{g} = \pm f^2 dt^2 + g_*$ is Einstein then $g_*$
is a metric of constant scalar curvature.
In particular $f$ is an eigenfunction of $\Delta_*$
for the eigenvalue $-nk_*$.
\end{cor}
In this context we mention without proof the following:
\begin{pro} {\rm (Lafontaine \cite[Thm.1.1]{La09})}

If a compact Riemannian $3$-manifold $(M_*,g_*)$ admits two linearly independent
solutions $f_1,f_2$ of Equation~\ref{eq:Corvino} (such that
$\overline{g} = f_i^2 dt^2 + g_*$ is Einstein for $i=1,2$) 
then $(M_*,g_*)$ is isometric to the standard $3$-sphere or
to a standard product $S^1 \times S^2$ or $S^1 \times \R P^2$.
\end{pro}

Notice the opposite sign convention in \cite{La09} for the Laplacian.

\medskip
Now we are going to describe nontrivial solutions of the equations
in \ref{Corvino}. It turns out that there are complete Einstein spaces
of this type with a complete $(M_*,g_*)$ 
that are not warped products with a 1-dimensional base themselves, see
Example~\ref{complete-example}.

\begin{pro}{\rm (The Ansatz of iterated warped products)}
\label{iterated}

Let $g_k$ be an Einstein metric with
$Ric= k(n-2) g_k, k \in \R$
on an $(n-1)$-dimensional manifold $M$ and 
$g_*=dx^2+u^2(x)g_k$
a warped product metric on $M_* = I \times M$ with an interval $I \subseteq \R$.
Let $f=f(x)$ be a smooth function on an interval in the real line.
Then Equation \ref{eq:Corvino}
%\begin{equation}
%f Ric_*-\nabla_*^2 f+ \Delta_* f g_*=0
%\end{equation}
holds if and only if $f(x)=a u'(x)$ for some constant $a\not=0$
and 
\begin{equation}
\label{eq:iterated}
u^2 u'''+ (n-3)u u'u''- (n-2)u'^3 + k (n-2) u' =0\,.
\end{equation}
Hence for a positive solution $u$ 
of Equation~\ref{eq:iterated}
the warped product metric 
$$\overline{g} = 
\pm u'^2(x) dt^2+ dx^2 + u^2(x)g_k
$$
on $I \times I \times M$
is an Einstein metric. Here $I$ is an interval on which
$u$ and $u'$ are positive.
The metric $\overline{g}$ can be regarded as a warped product 
$g_* + u'^2dt^2$ with an $n$-dimensional base and a $1$-dimensional fibre or as a warped product
$dx^2 \pm u'^2dt^2 + u^2g_k$ with a $2$-dimensional base in the (x,t)-plane. 
\end{pro}

Remark: The classical solutions $u(x) = x, \sin x, \cos x, e^x, 
\sinh x, \cosh x$
lead to Einstein warped product metrics $g_*$ for appropriate choices of $k$
with a 1-dimensional base. Especially $u(x) = x$ is a solution with $k=1$
corresponding to cylindrical coordinates in the Ricci flat $\R\times M_*$
where $M_*$ is also Ricci flat.
This is not interesting by \ref{Brinkmann}. 
We will see below in Proposition~\ref{classification} 
that there are in fact non-classical solutions leading
to complete Einstein metrics on $\R \times M_*$ with a non-Einstein space $(M_*,g_*)$. 

\begin{proof}
%In this proof we write $g$ instead of $g_*$. 
Since $\nabla_X \partial_x= \frac{u'}{u}X$ we obtain 
\begin{equation*}
\nabla^2f(\partial_x,\partial_x)= f'',\ \ \nabla^2f(X,X)=\frac{f'u'}{u} 
%\langle X,Y\rangle
\end{equation*}
where $X$ denotes a unit tangent vector orthogonal to $\partial_x,$ hence
\begin{equation*}
\Delta f=f''+(n-1)\frac{f'u'}{u}\,.
\end{equation*}
This leads to the equations
\begin{eqnarray*}
-\nabla^2f(\partial_x,\partial_x)+\Delta f 
&=&(n-1)\frac{f'u'}{u}\\
-\nabla^2f(X,X)+\Delta f 
&=& f''+(n-2)\frac{f'u'}{u}
\end{eqnarray*}
and 
%\begin{eqnarray}
%f {\rm Ric}\left(\partial_x,\partial_x\right)-\nabla^2f({\partial_x},\partial_x%)+\Delta f 
%=\frac{n-1}{u}\left(f'u'-f''u\right)\label{eq:fric1}\\
%f {\rm Ric} (X,X)-\nabla^2f(X,X)+\Delta f 
%=\\ 
%\left\{
%\frac{f}{u^2}
%\left(
%k(n-2)-(n-2)u'^2-u u''
%\right)+
%f''+(n-2)\frac{f'u'}{u}
%\right\} \label{eq:fric2}
%\end{eqnarray}

\begin{equation}
f {\rm Ric}\left(\partial_x,\partial_x\right)-\nabla^2f({\partial_x},\partial_x)+\Delta f 
=\frac{n-1}{u}\left(f'u'-f''u\right)\label{eq:fric1}
\end{equation}

\begin{equation}\label{eq:fric2}
f {\rm Ric} (X,X)-\nabla^2f(X,X)+\Delta f
\end{equation} 
$$= 
\frac{f}{u^2}
\left(
k(n-2)-(n-2)u'^2-u u''
\right)+
f''+(n-2)\frac{f'u'}{u}
$$
Assuming Equation~\ref{eq:Corvino} we see that the right hand side must vanish.
Hence we conclude from Equation~\ref{eq:fric1} that
$f'/f=u''/u$ and, therefore,
$f=au'$ for some constant $a.$ Without loss of generality let $a=1.$
From Equation~\ref{eq:fric2} we conclude Equation~\ref{eq:iterated}.
Conversely, Equation~\ref{eq:iterated} implies Equation~\ref{eq:fric2}  
and Equation~\ref{eq:Corvino}. \end{proof}
\begin{exa} \rm(particular non-standard solution)\label{non-standard}

Equation \ref{eq:iterated} is satisfied by the function $u(x) = x^{2/n}$
if $k=0$. In fact we have 
$$\textstyle u'(x) = \frac{2}{n}x^{(2-n)/n}, 
u''(x) = \frac{4-2n}{n^2}x^{(2-2n)/n}, 
u'''(x) = \frac{(4-2n)(2-2n)}{n^3}x^{(2-3n)/n}.$$
For $n=3$ we obtain the 4-dimensional Einstein metric
$\frac{4}{9}x^{-2/3}dt^2 + dx^2 + x^{4/3}g_0$
where $(M,g_0)$ is any flat 2-manifold. 
Obviously the sectional curvature in any $(t,x)$-plane is not constant,
so this is not a space of constant curvature.
Moreover, for $x \to 0$ we run into a singularity.
\end{exa}
\begin{cor}{\rm (Explicit integration)}
\label{explicit}

Equation \ref{eq:iterated} is equivalent to any of the following 
differential equations:
\begin{equation}\label{second}
uu'' + \frac{ n-2}{2}u'^2 + du^2 = k\frac{n-2}{2} \mbox{ for some constant } d
\end{equation}
\begin{equation}\label{first}
u^{n-2}\Big(u'^2 + \frac{2d}{n}u^2 -k \Big) = c\mbox{ for 
constants } c \mbox{ and } d
\end{equation}
\begin{equation}\label{eq:Besse}
u'^{2}  +\frac{2d}{n}u^{2} -cu^{2-n} = k  \mbox{ for 
constants } c \mbox{ and } d
\end{equation}
\begin{equation}\label{second-oscillator}
(u^{n/2})'' +\frac{dn}{2}u^{n/2} = k\cdot \frac{n(n-2)}{4}u^{(n-4)/2} 
\mbox{ for some constant } d
\end{equation}
\begin{equation}\label{first-oscillator}
\Big((u^{n/2})'\Big)^2 +\frac{dn}{2}\Big(u^{n/2}\Big)^2 = k\cdot \frac{n^2}{4}u^{n-2} + e\mbox{ for 
constants } d \mbox{ and } e
\end{equation}

\end{cor}
The two last equations are modified oscillator equations.
They can be explicitly solved by integrals of elementary functions.
For $k=0$ we obtain a classical oscillator equation for the function 
$(u(x))^{2/n}$.
The particular solution in Example \ref{non-standard} is the case $d=0$.
Equation \ref{eq:Besse} coincides with the version in \cite{Be87}.
Equation \ref{first} was used in \cite{Ku88} and \cite{KR97}.
Here the case $c=0$ leads to the classical solutions of $u'^2 + \frac{2d}{n}u^2 = k$. The case $c=0$ in \ref{first} is also the classical case
of the oscillator equation for $u$. In this case the
metric $dx^2 + u'^2dt^2$ on the 2-dimensional base is of constant curvature.
\begin{proof}We discuss only the case of a non-constant and non-vanishing
solution $u(x)$.

\medskip
The equivalence of Equation~\ref{eq:iterated} with Equation~\ref{second}:
We differentiate
$$\frac{u''}{u} + \frac{(n-2)u'^2}{2u^2} + d = k \frac{n-2}{2u^2}$$
and multiply by $u^3$. This leads to
$$u^2u''' - uu'u'' + (n-2)uu'u''-(n-2)u'^3 = -k(n-2)u'.$$
The equivalence of Equation~\ref{first} with Equation~\ref{second}:
We multiply Equation~\ref{second} by $2n^{n-2}u'$ and obtain
$$\big(u^{n-2}u'^2\big)' + \big( \frac{2d}{n}u^n -ku^{n-2}\big)' = 0.$$ 
The equivalence of  Equation~\ref{second} 
with Equation~\ref{second-oscillator}:
Evaluating the second derivative of $u^{n/2}$ leads to
$$\frac{n}{2}\Big(u^{(n-2)/2}u'' + \frac{n-2}{2}u^{(n-4)/2}u'^2   \Big).$$
Thus Equation~\ref{second-oscillator} takes the form
$$\frac{n}{2}u^{(n-4)/2}\Big(uu'' + \frac{n-2}{2}u'^2 + du^2   \Big) =
\frac{n}{2}u^{(n-4)/2}\cdot k\frac{n-2}{2}.$$
The equivalence of  Equation~\ref{second-oscillator} 
with Equation~\ref{first-oscillator} follows by differentiating the
first order equation and dividing by $2(u^{n/2})'$.
\end{proof}

\begin{exa}\rm \label{complete-example}
The simplest examples can be constructed from 
Equation~\ref{second-oscillator}
in the case $k=0$ and $d=2/n$. 
Let $(M,g)$ be a Ricci flat Riemannian manifold,
and let $u(x) = (\cosh(x))^{2/n}$. Then the metric
$u'^2dt^2 + dx^2 + u^2g$ is Einstein.

\medskip
In particular for complete $(M,g)$ this metric $dx^2 + (\cosh(x))^{4/n}g$ 
provides a complete non-standard solution (i.e., not Einstein) 
$(\R \times M,g_*)$ in Corollary~\ref{Corvino}.
If moreover $n=4$ and $k=-1/2$ then we have the solution
$u^2(x) = e^x +1$
of Equation~\ref{second-oscillator} with $u\not = 0$ 
and $u' \not =0$ everywhere. Therefore this leads to a complete
Einstein metric 
$$\frac{e^{2x}}{4(e^x+1)} dt^2 + dx^2 + (e^x + 1)g$$ on $\R \times M_* = \R^2 \times M.$
\end{exa}

\begin{cor}({\rm warped products of constant scalar curvature})

The following conditions are equivalent for a function $u(x)$
and a real constant $k$:
\begin{enumerate}
\item  $g_k$ is an Einstein metric with ${\rm Ric} = k(n-2)g_k$ on an 
$(n-1)$-manifold $M$ and
$\overline{g} = u'^2(x)dt^2 + dx^2 + u^2(x)g_k$ is Einstein.
\item  $g_{(k)}$ is a metric of constant normalized scalar curvature $k$ on an 
$(n-1)$-manifold $M$ and the warped product
$h_{(k)} = dx^2 + u^2(x)g_{(k)}$ is again a metric of constant scalar curvature.
\end{enumerate}
\end{cor}

\begin{proof} The first condition (1) is equivalent to Equation \ref{eq:Corvino}
which in turn is equivalent to Equation \ref{eq:iterated} and to
Equation \ref{second} by the calculations above.
On the other hand the second condition (2) is equivalent to the same
Equation \ref{second}, which is a standard formula in this case,
see the proof of Theorem 24 in \cite{Ku88} or Lemma 4.6 in \cite{KR97}. 
Here $2d/n$ is
the constant normalized scalar curvature of $h_{(k)}$. \end{proof}

By Corollary \ref{explicit} all local solutions can be expressed in terms
of elementary functions. Periodic solutions $u(x)$ of Equation \ref{second} 
lead to a countable number of solutions
on the compact manifold $S^1 \times S^{n-1}$, see Ejiri \cite{Ej82}.
The simplest one is $u(x) = \sqrt{2 + \cos x}$ for $n=4$, 
$d = \frac{1}{2}, k = 1$. Moreover, $u \frac{\partial}{\partial x}$ 
is a conformal (locally gradient) vector field \cite{KR97}.
\section{$4$-dimensional conformally Einstein products: 
Towards a classification}

Our Main Theorem~\ref{pseudoriem} provides a rough classification
of 4-dimensional conformally Einstein products:
These are either warped products or they are of the 
type $M \times N$ with two 2-manifolds $M,N$ and a particular conformal factor
$f(y,x) = a(y) + b(x)$ satisfying 
$ca^2 + 2\epsilon_1a'' = d =cb^2 + \epsilon_2b''$, compare the proof of 
the Main Theorem. 
In the latter case all possible metrics
are explicitly classified by the solutions of these two ODEs.
It remains to discuss the possible warped products.
We recall the following general results on Einstein warped products.

\begin{lem} {\rm (\cite[9.116]{Be87})} \label{Besse-Lemma}
A warped product metric $g_M = g_B + f^2g_F$ on $M = B \times F$
with a function $f \colon B \to \R$ and with
$p = {\rm dim}(F), q = {\rm dim}(B)$
is Einstein with ${\rm Ric}_M = \lambda_Mg_M$ 
if and only if the following three conditions hold:
\begin{enumerate}
\item $(F,g_F)$ is Einstein with ${\rm Ric}_F = \lambda_Fg_F$
\item ${\rm Ric}_B -\frac{p}{f}\nabla^2_Bf = \lambda_M g_B$
\item $f\Delta f + (p-1)||{\rm grad} f||^2 + \lambda_Mf^2 = \lambda_F$
\end{enumerate}
If $q=2$ then $(3)$ is equivalent to
Equation~\ref{second} with $p = n-1$, $f = u$,
$\lambda_M = 2d$ and $\lambda_F = k(n-2)$, and $(2)$ is equivalent with $(\nabla_B^2f)^\circ = 0$.
\end{lem}

\begin{proof}
The three equations are standard equations for warped products,
the so-called O'Neill-equations.
In the special case $q=2$ 
we have ${\rm Ric}_B = \lambda_Bg_B = K_Bg_B$ with the Gaussian curvature $K_B$,
$\lambda_F = K_F$ and, therefore,
$(\nabla^2_Bf)^0 = 0$ and $\nabla^2_Bf = f''g_B$.
This implies that $g_B$ itself is a warped product metric
$g_B = dt^2 + f'^2(t)dx^2$ with $K_B = -f'''/f'$.
In (3) we have $\Delta f = 2f''$ and $||{\rm grad} f|| = f'^2$.
Equation~\ref{second} follows which is equivalent with Equation~\ref{eq:iterated}. 
Conversely, from Equation~\ref{second}
we get back to (3) and from the statement of Proposition~\ref{iterated}
we get back to (1) and (2).
\end{proof}

\begin{cor} {\rm (Brinkmann \cite{Br25})}

Let $(M,g)$ be an Einstein space with a warped product metric
$g = \epsilon dt^2 + u^2(t) g_*$. Then $g_*$ is an Einstein metric and $u$ satisfies the ODE
from Corollary~\ref{Brinkmann}, i.e. 
$$u'' + \epsilon k u = 0, \ (u')^2 + \epsilon k u^2 = \epsilon k_*$$
Conversely, these conditions are also sufficient.
\end{cor}

\begin{cor} {\rm (1-dimensional base, Brinkmann \cite[Thm.III]{Br25})}

A $4$-dimensional Einstein warped product 
with a $1$-dimensional base is of constant sectional curvature.
\end{cor}

This follows from the proof of Corollary~\ref{Brinkmann} together with the fact
that a 3-dimensional Einstein fibre must be of constant
sectional curvature. This implies that the entire space is
of constant sectional curvature. For the Riemannian case compare
\cite[9.109]{Be87} or \cite{Ku88}.

\begin{cor}{\rm (2-dimensional base, A.Besse \cite[9.116-9.118]{Be87})}
\label{2-dimensional}

A $4$-dimensional Einstein warped product $g_M = g_B + u^2g_F$ 
with a $2$-dimensional base $B$
has a fibre $F$ of constant Gaussian curvature $K_F$ and a base $B$ of 
the following type:

The warping function $u$ on $B$ is a never vanishing solution of the equation 
$(\nabla_B^2u)^0 = 0$, hence the metric on $B$ is a warped product
$\epsilon dt^2 + u'^2(t)dx^2$ itself (an abstract surface of revolution, possibly
with exceptional points with $u' = 0$).
Moreover, $u$ satisfies the ODE
\begin{equation}\label{eq:2-base}
2\epsilon uu'' + \epsilon u'^2 + \lambda_M u^2 = K_F
\end{equation}
where $\lambda_M$ is the Einstein constant of the $4$-manifold
and $K_F$ is the constant Gaussioan curvature of $g_F$.
One integration step as in \ref{explicit} leads to the equation
\begin{equation}\label{eq:2-base-first}
u\Big(\epsilon u'^2 + \frac{\lambda_M}{3}\epsilon u^2 - K_F\Big) = c
\end{equation}
with a constant $c$. Here the coefficient $\lambda_M/3$ is nothing but the
normalized scalar curvature $k_M$ of $M$.
\end{cor}

%A classification of complete manifolds of this type
%is available in \cite[9.119]{Be87}.
%There is no compact $B$ of this type except for the unit sphere
%by the same argueing as in Corollary~\ref{compact} above:
%If $f$ is positive everywhere there is a positive minimum $x_1$ and a positive
%maximum $x_2$. Then from the ODE it is impossible that $f''(x_2) = -f''(x_1)$. 

\begin{proof}
%{\it Proof of \ref{2-dimensional}.}
In our special case $p=q=2$ we have ${\rm Ric}_B = \lambda_Bg_B = K_Bg_B$,
$\lambda_F = K_F$ and, therefore,
$(\nabla^2_Bf)^0 = 0$.
This implies that $g_B$ is a warped product metric
$g_B = \epsilon dt^2 + u'^2(t)dx^2$ itself with $\epsilon = \pm 1$, $K_B = -\epsilon u'''/u'$
and $\nabla^2_Bf = \epsilon u''g_B$.
Furthermore (2) implies the equation 
$K_Bg_B -2 \epsilon \frac{u''}{u} = \lambda_Mg_B$ or, equivalently,
\begin{equation}\label{2-2-Einstein}
-\epsilon uu''' - 2\epsilon u'u'' = uu'\lambda_M.
\end{equation}
%By using $(ff'')' = ff''' + f'f''$ this can be integrated as
%$$2ff'' + f'^2 + \lambda_Mf^2 = c  \ \ (\mbox{ constant) }.$$
From (3) we obtain  similarly 
$$2\epsilon uu'' + \epsilon u'^2 + \lambda_Mu^2 = \lambda_F = K_F,$$
and Equation~\ref{2-2-Einstein} is nothing but the derivative
of this equation. Since this is an equation only on $B$, it follows
that $K_F$ must be constant. Equation~\ref{eq:2-base} follows.

\medskip
A compact base $B$ must be a surface of genus zero with two
exceptional points with $f' = 0$, namely, minimum and maximum of the function 
$f$. In fact the function $f(x) =\cos x$ solves this ODE
with $K_F = 1$ and $\lambda_M = 3$. An example of a compact base ist 
$dx^2 + \sin^2 x \ dt^2$ in a warped product
$dx^2 + \sin^2 x \ dt^2 +  \cos^2 x \ g_F.$
If in this case $F$ is the unit 2-sphere then $M$ is the unit 4-sphere 
in cylindrical coordinates around an equator which is,
strictly speaking, a warped product only on a dense subset.
Here $B$ appears as an equatorial 2-sphere of $M$.
In any other case zeros of $f$ would lead to a contradiction.
See also Proposition~\ref{classification}. 
%\hfill $\Box$
\end{proof}

\medskip
Remark: As in the explicit integration above we can transform 
Equation~\ref{eq:2-base} into the following form:
$$(u^{3/2})'' + \frac{3\lambda_M}{4}fu^{3/2} = \frac{3K_F}{4}u^{-1/2}$$
which is again a modified oscillator equation for $f^{3/2}$.
We can integrate this equation leading to
$$((u^{3/2})')^2 + \frac{3\lambda_M}{4}(u^{3/2})^2 = \frac{9K_F}{4}u + c$$
with a constant $c$.

\begin{pro}  {\rm (3-dimensional base)}

$4$-dimensional Einstein warped products with a $3$-dimensional base
can be constructed along the lines of the iterated warped products in
Proposition~\ref{iterated}. These depend on several real parameters,
and in general these metrics are not of constant curvature.
Therefore, they cannot be described by warped products with a $1$-dimensional 
base.
\end{pro}
Remark: A complete classification of this case 
would depend on a complete classification
of the solutions of Equation~\ref{eq:Corvino} on 3-manifolds which does not
seem to be known. From Equation~\ref{eq:iterated} we just obtain special
families of examples that can also be written as warped products with
a 2-dimensional base.
More examples of open subsets $U$ of homogeneous spaces $M$ can be found
in \cite[9.122]{Be87}. 

\begin{cor}

$5$-dimensional Riemannian Einstein warped products with a $2$-dimen\-sional 
base are products
$B^2 \times F^3(K_F)$ of a surface with a $3$-manifold of constant 
curvature $K_F$ with
the metric $g_B = dx^2 + u'^2(x) dt^2$, $\overline{g} = g_B + u^2g_F$ 
where $u$ satisfies the ODE
$$2uu'' + 2u'^2 + \overline{k}u^2 = 2K_F$$
or, equivalently,
$$(u^2)'' + \overline{k}u^2 = 2K_F.$$
This is a standard oscillator equation for $u^2$.
A special case is $u^2(x) = 2 + \cos x$ with $\overline{k} = 1, K_F = 1$.
\end{cor}
Remark: There is no compact base $B$ of this type. The only candidate is a surface
of genus zero with the metric  $g_B = dx^2 + u'^2(x) dt^2$ where $u$ is periodic
and strictly positive with the same value of $|u''|$ at minimum and maximum.
Such a solution of the ODE above does not exist by
Lemma~\ref{droplemma}.
For the example $u(x) = \sqrt{2+\cos x}$ this can be seen from
$u'(0) = u'(\pi) = 0$ and $u''(0) = -1/(2\sqrt{3})$, $u''(\pi) = 1/2$.
This is the same phenomenon as for the extremal surfaces above, see Figure 1.
 
\begin{pro}  {\rm (complete Einstein warped products with a 2-dimensional base)}
\label{classification}

There are essentially three types of complete pseudo-Riemannian
Einstein warped products with a $2$-dimen\-sional base
and an $(n-1)$-dimensional fibre that cannot be written as warped products with a $1$-dimensional base
and, in particular, that are not of constant curvature.
In any case (up to sign) the metric in $n+1$ dimensions is of the type
$$\overline{g}= \pm dx^2 + u'^2(x)dt^2 + u^2(x)g_k$$ as in 
Proposition~\ref{iterated} with an $(n-1)$-dimensional Einstein metric $g_k$
with normalized scalar curvature $k$ 
(or metric of constant curvature $k$ if $n=3$) and a strictly positive function
$u(x)$ satisfying the equations in Corollary~\ref{explicit}.

\smallskip
Type I: $u$ is defined on $\R$ with $u>u_0>0, u'>0$ everywhere, asymptotically we have
$u(x) \sim u_0 + e^{ax}$ at $\pm \infty$ with positive constants $a, u_0$. 
The base is the $(x,t)$-plane or a cylindrical quotient of it.

\smallskip
Type II: $u'$ has a zero at $x_0$ with $u''(x_0) = 1$, and $u>0$ is defined on $[x_0, \infty)$
with an asymptotic growth at infinity like a linear function.
The base is a (rotationally symmetric but non flat)
plane in polar coordinates.

\smallskip
Type III: $u'$ has a zero at $x_0$ with $u''(x_0) = 1$, and $u>0$ is defined on $[x_0, \infty)$
with an asymptotic growth at infinity like an exponential function.
The base is a (rotationally symmetric but not flat) 
plane in polar coordinates.

\medskip
There is no compact base of this type except for the standard $2$-sphere.

\medskip
Conversely, for any given complete metric $g_k$ of arbitrary signature
all three types lead to complete warped products provided the base is Riemannian.

\end{pro}
Remark: For Riemannian metrics this classification is mentioned through 
the types (c), (a), (d) in \cite[9.118]{Be87}.
However, the proof refers to a preprint which apparently never appeared,
and the solutions are not given explicitly.
Moreover, Type (c) there is not quite convincing, as far as the constant 
coefficients are concerned. The discussion of this classification
in the appendix of \cite{HPW} is a bit sketchy, in particular on Type (c). 
\begin{proof}
The special expression for the metric follows from Lemma~\ref{Besse-Lemma} 
above. In the equations the Einstein constant appears as $2d$.
So we recall Equation~\ref{first} with $\overline{k} = \frac{2d}{n}$ 
(normalized scalar curvature of $\overline{g}$ ):
$$u^{n-2}\Big(u'^2 + \overline{k}u^2 -k \Big) = c \ \mbox{ for 
some constant } c$$
In the particular case $c=0$ 
we have the oscillator equation
$$u'^2 + \overline{k}u^2 -k = 0.$$ 
Up to scaling the standard solutions are

$u(x) = a\sin x + b \cos x$ if $\overline{k} = k = 1$,

$u(x) = \sinh x$ if $\overline{k} = -1, k=1$,

$u(x) = \cosh x$ if $\overline{k} = k = -1$,

$u(x) = e^x$ if $\overline{k} = -1, k=0$,

$u(x) = x$ if $\overline{k} = k = 0$.

\smallskip
Only the cases $u(x) = \cosh x$ and $u(x) = e^x$ lead to global warped products,
in the other cases there can be completions of a dense subset that is a warped product.

\smallskip
The solution $u = e^x$ leads to
the warped product $dx^2 + e^{2x}(dt^2 + g_k)$ with a 1-dimensional base.
This is complete whenever $g_k$ is complete and Riemannian, 
compare Case (b) in \cite[9.118]{Be87}. 
However, the global warped product $-dx^2 + e^{2x}(dt^2 + g_k)$ is not geodesically complete,
see \cite[7.41]{ON83}.

\smallskip
Similarly the Riemannian solution
$dx^2 + \sinh^2x \ dt^2 + \cosh^2x \ g_k$
is complete if and only if $g_k$ is complete. The base is the hyperbolic plane
in geodesic polar coordinates.
It is of Type III. If $g_k$ itself is the hyperbolic $(n-1)$-space then the entire space
is the hyperbolic $(n+1)$-space in geodesic normal coordinates around this $(n-1)$-subspace.
 
\smallskip
Similarly, the case $dx^2 + \cosh^2 x \ dt^2 + \sinh^2 x \ g_k$
is a warped product over the hyperbolic plane (with exceptional points 
for $x=0$).
However, there is a completion only if $g_k$ is the unit sphere.
This implies that $\overline{g}$ is the hyperbolic space in geodesic normal coordinates 
around a geodesic. This space can also be written
as a warped product with a 1-dimensional base.

\smallskip
This solution $u = x$ leads to
the Ricci flat warped product $dx^2 + dt^2 + x^2 g_k$ where 
the base is the 2-dimensional
Euclidean plane. In polar coordinates $(r,\phi)$ in the same plane 
this can be transformed into
the form $dr^2 + r^2 (d\phi^2 + \cos^2\phi \ g_k)$ with a 1-dimensional base
(with exceptional points for $x=0$). There is a completion 
only if $g_k$ is the standard unit sphere and, consequently, if $\overline{g}$
is the Euclidean space in cylindrical polar coordinates.

\smallskip
The last case $dx^2 + \sin^2 x \ dt^2 + \cos^2 x \ g_k$
is the only warped product with a compact 2-dimensional base, namely, the unit 2-sphere,
compare \cite{Ki02}.
There is a completion only if $g_k$ is the unit $(n-1)$-sphere and, consequently,
if $\overline{g}$ is the unit $(n+1)$-sphere in geodesic normal coordinates
around an equatorial circle.
 
\medskip
In the general case $c \not = 0$ we integrate Equation~\ref{first} 
as in \cite{KR97}: A zero of $u$ implies $c=0$, so we may
require the solution $u(x)$ to be strictly positive everywhere.
The equation 
$$\Big(\frac{du}{dx}\Big)^2 = u^{2-n}\big(c - \overline{k} u^n + ku^{n-2} \big)$$
for $u$ leads to the expression for the inverse function 
$$x(u) = \pm \int_{u_0}^u \sqrt{\frac{v^{n-2}}{c-\overline{k} v^n + kv^{n-2}}} \ dv.$$
By $u > 0$ there must be a positive zero of the polynomial 
$P(v) := c - \overline{k} v^n + kv^{n-2} $ in the denominator.
From the derivative $P'(v) = v^{n-3}\big(k(n-2) - \overline{k}nv^2 \big)$ 
with at most one positive zero
we see that $P$ has either at most two positive zeros of order one or
exactly one positive zero of order two 
(the zero $v=0$ of $P'$ is not relevant here since $P(0) = c \not = 0$).

\medskip 
{\sc Case 1:} There is one positive zero $u_0$ of order two, 
i.e., $P(u_0) = P'(u_0) = 0$ and $P(v) > 0$ for all $v > u_0$. 
Therefore $\overline{k} = 0$ would imply $k = 0$ and $P' \equiv 0$,
a contradiction. Hence we have 
$\overline{k} \not = 0$.
In this case the integral $\int_{u_0}^u$
is infinite for any $u > u_0$, and the integrand tends to 
$1/(\sqrt{-\overline{k}}v)$ for $v \to \infty$,
so the integral tends to $1/\sqrt{-\overline{k}}(\log u - \log u_0)$ for 
$u \to \infty$ in an asymptotic sense. 
In particular we have $\overline{k} < 0$. From $P'(u_0) = 0$ we see that 
$k = \overline{k}u_0^2n/(n-2) < 0$ as well. From $P(u_0)=0$
it follows that $c = - 2\overline{k}u_0^n/(n-2) > 0$.
For the inverse function $u(x)$ this means that $u(x)$ tends to $u_0$ for $x \to -\infty$
and that it is asymptotic to $\exp\big(\sqrt{|\overline{k}|}x\big)$ for $x \to \infty$.  
In particular $u(x) > u_0$ and $u'(x) > 0$ everywhere. 
This is Type I in the proposition.

\medskip
{\sc Case 2:} There is exactly one positive zero $u_0$ of order one, 
i.e., $P(u_0) = 0$
and $P(v) > 0$ for all $v > u_0$. Then the integral 
$\int_{u_0}^u$ converges 
for any $u > u_0$, and the integrand tends to 
$1/\sqrt{k}$ if $\overline{k} = 0$ and to
$1/(\sqrt{-\overline{k}}v)$ otherwise for $v \to \infty$.
Consequently the integral tends to $1/\sqrt{k}u$ for $u \to \infty$ 
if $\overline{k}=0$,
and the integral tends to $1/\sqrt{-\overline{k}}\big(\log u - \log u_0\big)$ 
for $u \to \infty$ 
in an asymptotic sense if $\overline{k} < 0$.
For the inverse function $u(x)$ this means that for some finite $x_0$ 
we have $u'(x_0) = 0$,
so that $dx^2 + u'^2dt^2$ describes a metric in the plane in geodesic 
polar coordinates provided that
$u''(x_0) = 1$, a condition that can be satisfied. 
This is Type II in the proposition for $\overline{k}=0$ and Type III 
for $\overline{k}<0$.

\medskip
{\sc Case 3:} There are precisely two positive zeros $u_1 < u_2$ of order one.
 For $u > u_2$ and $\overline{k} < 0$ we obtain the same as in Case 2 for $u > u_0$. For 
$u_1 < u < u_2$ and $\overline{k} > 0$ 
we see that the integral for $x(u)$  converges
on either side $u_1$ and $u_2$. This means that $u(x)$ 
is bounded and attains minimum and maximum
(actually it is periodic, Ejiri's solution \cite{Ej82}).
Therefore, if the base is complete and Riemannian, it must be compact.
However, a compact base is not possible by the following argument:
If $u'' = 1$ at the (positive) minimum then $u'' \not = -1$ 
at the (positive) maximum, as in the case
of the drop surfaces above, compare Corollary~\ref{compact}. This follows from
Lemma~\ref{droplemma}.
For the example $u(x) = 2\sqrt{2+ \cos x}$ with $n=4$ we have
$u'' = 1$ at the minimum $x = \pi$ and $u'' = -1/\sqrt{3}$
at the maximum $x = 0$.

The same argument shows that there is also no $C$-complete Lorentzian base
with infinitely many critical points of a bounded 
$u$ in the sense of \cite{KR95}.
If one starts at one critical point with $u'' = 1$ then the next critical
point turns out to be a singularity since $u'' \not = -1$:
The metric is smooth at every second critical point, 
the other ones are singularities. 

\medskip
For the converse direction we assume that the 2-dimensional base has a positive definite metric,
and $u(x)$ is a stritly positive solution according to the cases I, II, III.
Then the base is complete.
In particular $u$ is bounded below, that it, $u(x) \geq C > 0$ for a constant $C$.
We may assume $C \leq 1$.
If $u(x) \leq 1$ then $u/\sqrt{1+u^2} \geq u/2 \geq C/2$. If $u(x) > 1$ then
$u/\sqrt{1+u^2} \geq 1/2 \geq C/2$, so we obtain $u/\sqrt{1+u^2} \geq C/2$
in any case. It follows that the integral of $u/\sqrt{1+u^2}$ is unbounded
on either side of the real line.
Then the completeness of the warped product follows from Theorem 3.40 in \cite{CS}.
For the case of a Lorentz base the situation is more complicated, see \cite{CS}.
\end{proof}

\begin{lem}{\rm (compare \cite[Sec.2]{Ki02} with a similar proof)}\label{droplemma}

The differential equation
$$u'^2=\alpha-\beta u^2+\gamma u^{2-n}$$
with arbitrary $\alpha,\beta, \gamma \in \R$ and an integer $n\ge 3$
does not admit any solution $u:I \rightarrow \R$ on
an open interval $I \subseteq \R$ that attains its positive minimum $a$ and its
maximum $b$ such that
%$I\subset \R$ with $a,b \in I,$ $A=u(a)=\min \{u(x), x\in I\}>0,
%B=u(b)=\max \{u(x), x\in I\}$ i.e.
$u'(a)=u'(b)=0$ and $u''(a)=1, u''(b)=-1.$
\end{lem}

\begin{proof}
We assume that $A = u(a) > 0$ is the minimum value of the function
and $B = u(b) > A$ is the maximum value such that $u''(a) = 1, u''(b) = -1.$
First of all $\gamma = 0$ would imply that either $u$ has a zero 
(if $\beta \geq 0$) or $u$
is unbounded (if $\beta < 0$) which
is impossible by assumption. So we have $\gamma \not = 0$.
Differentiating
\begin{equation*}
\label{eq:first}
u'^2=\alpha-\beta u^2+\gamma u^{2-n}
\end{equation*}
yields
\begin{equation*}
\label{eq:second}
2 u''=- 2\beta u-(n-2)\gamma u^{1-n}\,.
\end{equation*}
Now $\beta = 0$ would imply that $u''$ cannot change its sign which
is impossible by assumption. Therefore we have $\beta \not = 0$.
We define the function
\begin{equation*}
g(x)=\alpha -\beta x^2+\gamma x^{2-n}=
x^{2-n}\big[x^{n-2}(-\beta x^2+\alpha)+\gamma\big].
\end{equation*}
With
\begin{equation*}
g'(x)=-2\beta x-(n-2)\gamma x^{1-n}
\end{equation*}
and the notations $u'(x)=g(u(x))\, , \,u''(x)=g'(u(x))/2$
the equation $g(A)=0$ implies
\begin{equation}
\label{eq:one}
\gamma=A^{n-2}\left(\beta A^2-\alpha\right).
\end{equation}
Then from $g'(A)=2 u''(A)=2$ one obtains
\begin{equation}
\label{eq:two}
\gamma=-\frac{2A^{n-1}}{n-2}\left(\beta A +1\right)\,.
\end{equation}
Equation~\ref{eq:one} and Equation~\ref{eq:two}
imply:
\begin{equation}
\label{eq:Aeins}
-\frac{2A^{n-1}}{n-2}\left(\beta A +1\right)=\gamma =A^{n-2}\left(\beta A^2-\alpha\right).
\end{equation}
Since $A \not =0$ this is eqivalent with the quadratic equation
\begin{equation*}
A^2+\frac{2}{n\beta} A-\frac{n-2}{n \beta}\,\alpha=0.
\end{equation*}
Since $A>0$ by assumption we obtain
$\alpha\beta>0$ (in particular $\alpha \not = 0$) and
\begin{equation}
\label{eq:A}
A=\frac{\sqrt{1+(n-2)n\alpha\beta} -1}{n \beta}
\end{equation}
Analogously 
\begin{equation}
\label{eq:Beins}
\gamma=B^{n-2}\left(\beta B^2-\alpha\right)=
\frac{2 B^{n-1}}{n-2}\left(1-\beta B\right)
\end{equation}
hence
\begin{equation*}
B^2-\frac{2}{n\beta} B-\frac{n-2}{n\beta}\, \alpha=0
\end{equation*}
Since $B>A>0$ and $\alpha \beta>0$
we obtain:
\begin{equation}
\label{eq:B}
B=\frac{\sqrt{1+(n-2)n\alpha\beta} +1}{n \beta}
\end{equation}
Equation~\ref{eq:Aeins} and Equation~\ref{eq:Beins}
imply
\begin{equation*}
\frac{A^{n-1}}{B^{n-1}}=\frac{\beta B-1}{\beta A+1}
\end{equation*}
With the notation 
\begin{equation*}\label{eq:y}
y:=\sqrt{1+(n-2)n\alpha\beta}>1
\end{equation*}
we obtain from Equation~\ref{eq:A} and Equation~\ref{eq:B}:
\begin{equation*}
\left(\frac{y+1}{y-1}\right)^{n-1}=\frac{y+(n-1)}{y-(n-1)}
\end{equation*}
resp.:
The positive value $y$ is a zero of the polynomial
\begin{equation*}
\phi_m(x):=\left(x-1\right)^{m}\left(x+m\right)
-\left(x+1\right)^{m}\left(x-m\right)
\,.
\end{equation*}
for $m=n-1.$
Since for $m\ge 1:$
\begin{equation*}
\phi_m'=(m+1)\phi_{m-1}
\end{equation*}
and 
\begin{equation*}
\phi_m(0)=\left\{
\begin{array}{ccc}
2m &;& m \mbox{ even}\\
0 &;& m \mbox{ odd}
\end{array}
\right.
\end{equation*}
we obtain 
that 
\begin{equation*}
\phi_m^{(k)}(0)=
\left\{
\begin{array}{ccc}
2m (m+1)\cdots(m+2-k) &;& (m-k) \mbox{ even}\\
0 &;& m \mbox{ odd}
\end{array}
\right.
\end{equation*}
Hence all coefficients $a_{k}$ of the polynomial
$\phi_m(x)=\sum_{k=0}^{m-2} a_k x^k:$
are non-negative and the coefficients $a_k$ with
$0\le k\le n-2$ and $m \equiv k \pmod{2}$ are positive.
Therefore the polynomial $\phi_m$ has no positive zero.
This contradiction finishes the proof.
\end{proof}
%%%%%%%%%%%%%%%%%%%%%%%%%%%%%%%%%%%%%%%%%%%%%%%%%%%%%%%%%%%%%%%%%

\bigskip
{\small Wolfgang K\"uhnel\\ Institut f\"ur Geometrie und Topologie\\
Universit\"at Stuttgart,
D-70550 Stuttgart\\
{\tt kuehnel@mathematik.uni-stuttgart.de}

\smallskip
Hans-Bert Rademacher\\ Mathematisches Institut\\
Universit\"at Leipzig, 
D-04081 Leipzig\\
{\tt rademacher@math.uni-leipzig.de}}

\end{document}